\documentclass{scrartcl}

\usepackage{amsmath, amssymb, amsthm, marginnote, ltxcmds, adjustbox, stmaryrd, graphicx, bbold, extarrows}
\usepackage{esvect}

\usepackage{mathtools, stackengine, changepage, ragged2e}

\usepackage{todonotes}
\usepackage{float}
\usepackage{multirow}

\definecolor{cite}{HTML}{11871E}
\definecolor{url}{HTML}{698996}
\definecolor{link}{HTML}{912F1B}

\usepackage[pdfencoding=unicode, colorlinks=true, linkcolor=link, citecolor=cite, urlcolor=url, linktocpage]{hyperref}

\usepackage{subfig}

\usepackage[backend=biber, style=alphabetic, maxnames=10, maxalphanames=3, minalphanames=3]{biblatex}
\usepackage{tikz}
\usetikzlibrary{cd}
\usetikzlibrary{arrows, arrows.meta, positioning, calc}
\tikzcdset{arrow style=tikz, diagrams={>={Straight Barb[scale=0.8]}}}

\tikzstyle{arrow} = [-{Straight Barb[scale=0.8]}, line width=0.2mm]
\tikzset{
math to/.tip={Glyph[glyph math command=rightarrow]},
loop/.tip={Glyph[glyph math command=looparrowleft, swap]},
}

\usepackage{enumitem}
{\end{enumerate}%
}

\usepackage[
	letterpaper,
	twoside=false,
	textheight=22cm,
	textwidth=14.4cm,
	marginparsep=0.75cm,
	marginparwidth=2.5cm,
	heightrounded,
	centering
]{geometry}

\usepackage{lineno}

\usepackage{libertine}
\usepackage{stmaryrd}

\usepackage[capitalise]{cleveref}

\crefalias{diagram}{equation}
\crefname{diagram}{diagram}{diagrams}

\Crefname{prop}{Proposition}{Propositions}
\Crefname{lem}{Lemma}{Lemmas}
\Crefname{cor}{Corollary}{Corollaries}
\Crefname{thm}{Theorem}{Theorems}
\Crefname{alphThm}{Theorem}{Theorems}

\Crefname{defn}{Definition}{Definitions}
\Crefname{notation}{Notation}{Notations}
\Crefname{cons}{Construction}{Constructions}
\Crefname{rmk}{Remark}{Remarks}
\Crefname{obs}{Observation}{Observations}
\Crefname{trick}{Trick}{Tricks}
\Crefname{warning}{Warning}{Warnings}
\Crefname{conj}{Conjecture}{Conjectures}
\Crefname{assump}{Assumption}{Assumptions}
\Crefname{recollect}{Recollection}{Recollections}
\Crefname{terminology}{Terminology}{Terminologies}
\Crefname{conditionsec}{Condition}{Conditions}
\Crefname{fact}{Fact}{Facts}

\Crefname{question}{Question}{Questions}
\Crefname{example}{Example}{Examples}

\Crefname{figure}{Figure}{Figures}

\crefformat{equation}{(#2#1#3)}
\crefformat{section}{\S#2#1#3}
\crefmultiformat{section}{\S\S#2#1#3}{and~#2#1#3}{, #2#1#3}{, and~#2#1#3}

\newtheorem{thm}{Theorem}[section]
\newtheorem{prop}[thm]{Proposition}
\newtheorem{lem}[thm]{Lemma}
\newtheorem{cor}[thm]{Corollary}

\newtheorem{alphThm}{Theorem}

\newcommand{\neutralize}[1]{\expandafter\let\csname c@#1\endcsname\count@}
\makeatother

\newtheorem*{thm*}{Theorem}
\newtheorem*{prop*}{Proposition}
\newtheorem*{lem*}{Lemma}
\newtheorem*{cor*}{Corollary}
  
\newtheorem{alphConj}{Conjecture}



\newtheorem{alphCor}[alphThm]{Corollary}



\newtheorem{alphProp}{Proposition}



\theoremstyle{definition}
\newtheorem*{defn*}{Definition}
\newtheorem{defn}[thm]{Definition}
\newtheorem{cons}[thm]{Construction}

\newtheorem{recollect}[thm]{Recollections}

\newtheorem{example}[thm]{Example}
\newtheorem{rmk}[thm]{Remark}
\newtheorem{obs}[thm]{Observation}
\newtheorem{constr}[thm]{Construction}

\newcommand{\spc}{\mathcal{S}}
\DeclareMathOperator{\ind}{\mathrm{Ind}}

\newcommand{\sphere}{\mathbb{S}}

\newcommand{\spectra}{\mathrm{Sp}}

\newcommand{\id}{\mathrm{id}}

\DeclareMathOperator{\coev}{coev}

\DeclareMathOperator{\im}{Im}
\DeclareMathOperator{\fib}{fib}
\DeclareMathOperator{\cofib}{cofib}

\DeclareMathOperator{\laxeq}{laxeq}
\newcommand{\artimes}{\overrightarrow{\times}}
\DeclareMathOperator{\equaliser}{eq}
\newcommand{\eval}{\mathrm{ev}}
\DeclareMathOperator{\pr}{pr}

\DeclareMathOperator{\map}{Map}
\DeclareMathOperator{\func}{Fun}
\DeclareMathOperator{\Hom}{Hom}
\DeclareMathOperator{\End}{End}

\DeclareMathOperator{\Nil}{Nil}
\DeclareMathOperator{\Endspec}{\mathrm{end}}
\DeclareMathOperator{\aut}{Aut}

\newcommand{\Endbar}{\overline{\mathrm{End}}}
\newcommand{\Nilbar}{\overline{\mathrm{Nil}}}
\newcommand{\trivbar}{\overline{\triv}}
\newcommand{\NKbar}[1]{\overline{N#1}}

\newcommand{\onecategory}[1]{\mathrm{#1}}
\newcommand{\cat}{\onecategory{Cat}}
\newcommand{\lpresentable}{\onecategory{Pr}^{\mathrm{L}}}
\newcommand{\rpresentable}{\onecategory{Pr}^{\mathrm{R}}}

\newcommand{\catperf}{\cat^{\mathrm{perf}}}
\newcommand{\category}[1]{\mathcal{#1}} %typeset for categories later on.

\newcommand{\op}{^{\mathrm{op}}}

\DeclareMathOperator{\alg}{Alg}
\newcommand{\module}{\mathrm{Mod}}

\newcommand{\tensoralg}[2]{T_{#1}(#2)}
\newcommand{\tensoralgaut}[2]{T_{#1}(#2)[{#2}^{-1}]}

\newcommand{\Dperf}[1]{\mathrm{D}^{\mathrm{perf}}}

\def\colim{\qopname\relax m{colim}}

\newcommand{\bbZ}{\mathbb{Z}}

\newcommand{\bbQ}{\mathbb{Q}}
\newcommand{\bbN}{\mathbb{N}}
\newcommand{\bbA}{\mathbb{A}}

\newcommand{\TC}{\mathrm{TC}}

\DeclareMathOperator{\shift}{shift}
\DeclareMathOperator{\triv}{triv}

\newcommand{\twistedpower}[2]{{#1}^{(#2)}}

\DeclareMathOperator{\fgt}{fgt}
\newcommand{\shifttransformation}[1]{[#1]}
\newcommand{\loc}[1]{\mathrm{loc}_{#1}}
\newcommand{\freeend}[1]{\mathrm{free}_{#1}}
\newcommand{\freeendbar}[1]{\overline{\mathrm{free}}_{#1}}

\newcommand{\arrdisp}{0.33ex}
\newcommand{\arrdisplacementsp}{0.72ex}

\newcommand{\ardis}{\ar@<\arrdisp>}
\newcommand{\ardissp}{\ar@<\arrdisplacementsp>}



\title{A twisted Bass-Heller-Swan decomposition for localising invariants}
\author{\textsc{Dominik Kirstein}\thanks{kirstein@mpim-bonn.mpg.de} \and \textsc{Christian Kremer}\thanks{kremer@mpim-bonn.mpg.de}}
\date{\today}

\begin{document}
\maketitle

\begin{abstract}
    We generalise the classical Bass-Heller-Swan decomposition for the $K$-theory of (twisted) Laurent algebras to a splitting for general localising invariants of certain categories of twisted automorphisms.
    As an application, we obtain splitting formulas for Waldhausen's $A$-theory of mapping tori and for the $K$-theory of certain tensor algebras.
    We identify the $\Nil$-terms appearing in this splitting in two ways. Firstly, as the reduced $K$-theory of twisted endomorphisms. Secondly, as the reduced $K$-theory of twisted nilpotent endomorphisms.
    Finally, we generalise classical vanishing results for $\Nil$-terms of regular rings to our setting.
\end{abstract}

\tableofcontents 

\section{Introduction}

The algebraic $K$-theory of polynomial- and Laurent extensions has been an object of interest since the very beginnings of the subject.
The fundamental theorem of algebraic $K$-theory, sometimes also called the Bass-Heller-Swan decomposition, states that for a ring $R$ there is a splitting
\begin{equation*}
    K_n(R[t, t^{-1}]) \simeq K_n(R) \oplus K_{n-1}(R) \oplus NK^+_n(R) \oplus NK^-_n(R).
\end{equation*}
This result dates back to the very beginnings of $K$-theory.
It was first established for regular rings and $n=1$ by Bass-Heller-Swan \cite{BHS} and Quillen later proved the general version for connective $K$-theory \cite{Grayson76}. The numerous applications of algebraic $K$-theory to geometric topology, such as the $s$-cobordism theorem or Farrell's fibering theorem \cite{Farrell71}, allowed to extract concrete geometric applications.
The splitting also served as source for the first constructions of negative algebraic $K$-theory.

There have been various generalisations of this result in two directions, and we give an incomplete list here.
In one direction, one replaces the ring $R$ by a homotopy coherent version.
H\"uttemann-Klein-Vogell-Waldhausen-Williams \cite{HKWWW01} prove an $A$-theoretic splitting for products $X \times S^1$, which can be thought of as a splitting for the Laurent ring spectrum $\sphere[\Omega X][t, t^{-1}]$.
Fontes-Ogle \cite{FontesOgle18} prove a version for connective $\sphere$-algebras, H\"uttemann \cite{Huettemann21} proves a version for $\bbZ$-graded rings.
Most recently, Saunier \cite{Saunier22} establishes such a splitting for general localising invariants of stable $\infty$-categories.
In a different direction, generalisations allow for twisted Laurent extensions.
This was established by Grayson \cite{Grayson88} for rings.
Waldhausen \cite{Waldhausen78} proves a version for his generalised Laurent extension and L\"uck-Steimle \cite{LueckSteimle16} prove a version for additive categories.
These results, combined with the Farrell-Jones conjecture, provide a powerful tool for computational and qualitative results about the algebraic $K$-theory of group rings \cite{LueckSteimle16II}. 

The goal of this work is to provide a common generalisation of most of the previously mentioned results.
We prove a general splitting result for the $K$-theory of certain categorical mapping tori, with the essential ingredient being Land-Tamme's work on the $K$-theory of pushouts \cite{LandTamme23}.
In contrast to most of the previous work, we also allow twists by noninvertible endomorphisms.
We work in the setting of localising invariants of stable $\infty$-categories as pioneered by Blumberg-Gepner-Tabuada \cite{BGT13}.
This immediately proves the splitting not only for nonconnective $K$-theory but also other localising invariants like topological Hochschild homology and its cousins.

\subsection*{Main results}

Let $\category{C}$ be an idempotent complete stable $\infty$-category together with an exact functor $\alpha \colon \category{C} \to \category{C}$.
Denote by $\category{C}_{h \bbN}$ its mapping torus, i.e. the pushout of the span $\category{C} \xleftarrow{\id \oplus \alpha} \category{C} \oplus \category{C} \xrightarrow{\id \oplus \id} \category{C}$ in $\catperf$, the $\infty$-category of idempotent complete stable $\infty$-categories.
In the case where $\alpha$ is an equivalence, this mapping torus (up to ignoring certain finiteness conditions) consists of pairs $(x,f)$ of an object $x \in \category{C}$ together with an equivalence $f \colon x \xrightarrow{\simeq} \alpha^{-1}(x)$.
We also call $\category{C}_{h \bbN}$ the category of \textit{twisted automorphisms}.
The general description of $\category{C}_{h \bbN}$ can be found in  \cref{def:twisted_endomorphism,lem:identification_mapping_torus_twisted_automorphisms}.
\begin{alphThm}\label{thmintro:bhs_decomposition}
    Let $\alpha \colon \category{C} \to \category{C}$ be an exact endofunctor of an idempotent complete stable $\infty$-category and $E \colon \catperf \to \category{E}$ a localising invariant.
    The assembly map $E(\category{C})_{h \bbN} \to E(\category{C}_{h \bbN})$ has a natural splitting
    \begin{equation*}
        E(\category{C}_{h \bbN}) \simeq E(\category{C})_{h \bbN} \oplus NE_\alpha(\category{C}) \oplus \NKbar{E}_{\alpha}(\category{C}).
    \end{equation*}
\end{alphThm}
The mapping torus $E(\category{C})_{h \bbN}$ fits into the cofiber sequence $E(\category{C}) \xrightarrow{\id - \alpha} E(\category{C}) \to E(\category{C})_{h \bbN}$.
If $E$ takes values in spectra, one obtains an associated long exact sequence of homotopy groups relating the homotopy groups of $E(\category{C})$ and $E(\category{C})_{h \bbN}$.

Let us explain the $\Nil$-terms appearing in this decomposition.
There is a category $\End_\alpha(\category{C})$ of twisted endomorphisms, which consists of pairs $(x,f)$ of an object $x \in \category{C}$ together with a morphism $f \colon x \to \alpha(x)$.
Denote by $\Nil_\alpha(\category{C}) \subseteq \End_\alpha(\category{C})$ the subcategory generated by the trivial endomorphisms $(x, 0 \colon x \to \alpha(x))$ under finite colimits and retracts.
If $\alpha$ is an equivalence, by \cref{prop:characterisation_nilpotent_endomorphism} these are precisely the homotopy nilpotent endomorphisms $(x,f)$ such that the composite $x \xrightarrow{f} \alpha(x) \xrightarrow{\alpha(f)} \alpha^2(x) \to \dots \to \alpha^n(x)$ is trivial for large $n$.
The inclusion $\triv \colon \category{C} \to \Nil_\alpha(\category{C})$ of trivial endomorphisms admits a retraction and $NE_\alpha(\category{C})$ is defined by the splitting
\begin{equation}\label{eq:splitting_k_theory_nil}
    E(\Nil_\alpha(\category{C})) \simeq E(\category{C}) \oplus \Omega NE_\alpha(\category{C}).
\end{equation}
The description of $\NKbar{E}_{\alpha}(\category{C})$ is similar using twisted endomorphisms $\alpha(x) \to x$ instead.

The $\Nil$-terms in \cref{eq:splitting_k_theory_nil} also arise in a different context.
There is the category $\End_{\alpha^R}(\ind(\category{C}))^\omega$ whose objects, up to a finiteness condition, are again pairs $(x, f \colon \alpha(x) \to x)$.
To any object $x \in \category{C}$ one can associate the free twisted endomorphism
\begin{equation*}
    \freeend{}(x) = \left(\bigoplus_{n \ge 0} \alpha^n(x), \shift\right),
\end{equation*}
where $\shift$ denotes the composition $\alpha(\bigoplus_{n \ge 0} \alpha^n(x)) \simeq \bigoplus_{n \ge 1} \alpha^n(x) \rightarrow \bigoplus_{n \ge 0} \alpha^n(x)$.
The map $\freeend{} \colon \category{C} \to \End_{\alpha^R}(\ind(\category{C}))^\omega$ again admits a retraction.
Our second main result identifies its cofiber with the $\Nil$-term from above.
\begin{alphThm}\label{thmintro:splitting_twisted_endomorphisms}
    The functor $\freeend{}$ induces a splitting
    \begin{equation*}
        E(\End_{\alpha^R}(\ind(\category{C}))^\omega) \simeq E(\category{C}) \oplus NE_\alpha(\category{C}).
    \end{equation*}
\end{alphThm}
There is an analogous splitting for the $\NKbar{E}_\alpha$-term.
This result goes back to Waldhausen \cite[Theorem 13.5]{Waldhausen78} in the setting of generalised polynomial extensions of discrete rings.
Land-Tamme \cite[Corollary 4.5, Proposition 4.7]{LandTamme23} prove a version for tensor algebras.
Our result is a generalisation to the case of categories not generated by a single object.

Using a recent dévissage result of Burklund-Levy \cite{BurklundLevy23}, we can prove vanishing of $\Nil$-terms under regularity assumptions, which generalises the classical vanishing result for regular rings.
\begin{alphCor}\label{thmintro:regularity}
    Suppose that $\alpha \colon \category{C} \to \category{C}$ is a left $t$-exact endofunctor of a stable category with a bounded $t$-structure.
    Then the connective $\Nil$-term $\tau_{\ge 0} NK_\alpha(\category{C}) \simeq 0$ vanishes.
    If $\category{C}^\heartsuit$ is Noetherian, then also the nonconnective $\Nil$-term $NK_\alpha(\category{C}) \simeq 0$ vanishes.
\end{alphCor}
Analogous vanishing results hold for $\NKbar{K}_\alpha$ if we instead require $\alpha$ to be right $t$-exact.

\subsection*{Applications}

For a ring spectrum $R$ and a $(R,R)$-bimodule $M$ denote by $T_R(M) = \bigoplus_{n \ge 0} M^{\otimes_R^n}$ its tensor algebra.
Suppose that $M$ admits a right dual $M^\vee$.
Define the localised tensor algebra by
\begin{align*}
    \tensoralgaut{R}{M} \simeq \colim\left(\bigoplus_{n \ge 0} M^{\otimes_R^n} \xrightarrow{\coev} \bigoplus_{n \ge 0} M^\vee \otimes_R M^{\otimes_R^n} \xrightarrow{\coev} \dots \right).
\end{align*}
The bimodule $M$ induces an endomorphism $M \otimes_R - \colon \module_R^\omega \to \module_R^\omega$ of the category of perfect left $R$-modules.
It turns out that $(\module_R^\omega)_{h \bbN} \simeq \module_{\tensoralgaut{R}{M}}^\omega$.
\cref{thmintro:bhs_decomposition} then reduces to the following result.
\begin{alphThm}\label{thmintro:splitting_localised_tensor_algebra}
    There is a splitting
    \begin{equation*}
        E(\tensoralgaut{R}{M}) \simeq E(R)_{h \bbN} \oplus NE_M(R) \oplus \NKbar{E}_M(R).
    \end{equation*}
\end{alphThm}
In the case where $R$ is discrete and $M$ is the bimodule $R$ with trivial left $R$-module structure and right $R$-module structure coming from a pure embedding $\alpha \colon R \xrightarrow{} R$, the localised tensor algebra $\tensoralgaut{R}{M} \simeq R_\alpha\{t^{\pm 1}\}$ identifies with Waldhausen's generalised Laurent extension \cite{Waldhausen78}.
It is the universal ring containing $R$ and an invertible element $t$ which satisfies $t r = \alpha(r) t$.
If $\alpha$ is an isomorphism, this is the classical ring of twisted Laurent polynomials given by $\bigoplus_{n \in \bbZ} R t^n$ with multiplication $rt^m \cdot st^n = r \alpha^m(s) t^{m+n}$.

As another application, we can use \cref{thmintro:bhs_decomposition} to obtain the following splitting for Waldhausen's (finitely dominated) $A$-theory of mapping tori.
\begin{alphThm}\label{thmintro:A_theory_splitting}
    Let $\alpha \colon X \to X$ be a selfmap of a space.
    Then there is a splitting
    \begin{equation*}
        A(X_{h \bbN}) \simeq \tau_{\ge 0}(\bbA(X)_{h \bbN}) \oplus NA_\alpha(X) \oplus \NKbar{A}_{\alpha}(X).
    \end{equation*}
\end{alphThm}
We actually prove a version of this for a nonconnective delooping $\bbA$ of $A$-theory and obtain \cref{thmintro:A_theory_splitting} by passing to connective covers.
In particular, on $1$-connective covers one has $\tau_{\ge 1}(\bbA(X)_{h \bbN}) \simeq \tau_{\ge 1}(A(X)_{h \bbN})$.
We also provide a guide to computations of the $A$-theoretic $\Nil$-terms by envoking the work of B\"okstedt-Hsiang-Madsen \cite{Bokstedt-Hsiang-Madsen} on the topological cyclic homology of spaces.

\subsection*{Outline of the proof}

Most of the results in the literature only consider the case where the twist $\alpha$ is an equivalence.
The proof usually follows the classical algebraic geometric approach and first establishes a splitting for versions of the (twisted) projective line over $R$ glued from (twisted) polynomial algebras $R_\alpha[t]$ and $R_{\alpha^{-1}}[t]$ along $R_\alpha[t^{\pm 1}]$.
This argument can not be extended to our setting of noninvertible twists.
We follow an approach which is more similar to Waldhausen's splitting for generalised Laurent extensions \cite{Waldhausen78}.

The main tool we use in the proof is Land-Tamme's theory of $K$-theory of pushouts \cite[Theorem 3.2]{LandTamme23}.
Starting from the pushout square
\begin{equation*}
\begin{tikzcd}
    \category{C} \oplus \category{C} \ar[r, "\id \oplus \id"] \ar[d, "\id \oplus \alpha"] & \category{C} \ar[d]\\
    \category{C} \ar[r] & \category{C}_{h \bbN}
\end{tikzcd}
\end{equation*}
defining the categorical mapping torus, Land-Tamme show how to replace the upper left corner by a different category $\im(\category{C} \oplus \category{C})$ such that the resulting commutative square becomes a pushout after applying the localising invariant $E$.
We show that $\im(\category{C} \oplus \category{C})$ admits an orthogonal decomposition through the two $\Nil$-categories $\Nil_\alpha(\category{C})$ and $\Nilbar_\alpha(\category{C})$.
By employing the splitting \cref{eq:splitting_k_theory_nil} and carefully analysing the resulting pushout square, we arrive at the claimed splitting.

\subsection*{Structure of the article}

We begin by recalling the necessary background on Land-Tamme's $K$-theory of pushouts in \cref{sec:K_theory_pushouts}.
Next, in \cref{sec:twisted_automorphisms,sec:nilpotent_endomorphisms} we introduce the categories of twisted endomorphisms, automorphisms and nilpotent endomorphisms and study their basic properties, which will be needed later on.
\cref{sec:proof_decomposition} contains a proof of \cref{thmintro:bhs_decomposition} and \cref{sec:k_theory_twisted_endomorphisms} constains a proof of \cref{thmintro:splitting_twisted_endomorphisms}.
In \cref{sec:regularity} we construct a $t$-structure on the category of twisted endomorphisms and prove \cref{thmintro:regularity}.
In the first half of \cref{sec:applications} we apply these results to obtain splittings for the $K$-theory tensor algebras and various kinds of (twisted) polynomial rings, as well as for $A$-theory.
In the second half \cref{sec:NA_calculation} we express some $A$-theoretic $\Nil$-terms through free loop spaces.

\subsection*{Conventions}

This article is written in the language of $\infty$--categories as set down in \cite{HTT,HA}, and so by a \textit{category} we will always mean an $\infty$--category unless stated otherwise.
We also use the following notations throughout:
\begin{itemize}
    \item $\cat$ denotes the category of small categories and $\spc \subseteq \cat$ the category of spaces.
    \item $\catperf$ denotes the category of small idempotent complete stable categories (sometimes called perfect categories) and exact functors.
    \item We denote by $\Hom_\category{C}(x,y)$ the mapping space between objects in a category $\category{C}$.
    If $\category{C}$ is stable, $\hom_\category{C}(x,y)$ denotes their mapping spectrum.
    \item $E \colon \catperf \to \category{E}$ denotes a localising invariant with values in a stable category.
\end{itemize}

\subsection*{Acknowledgements}

We are grateful to Andrea Bianchi, Markus Land, Wolfgang L\"uck, Maxime Ramzi, Victor Saunier and Christoph Winges for numerous helpful conversations and encouragements on this project.
This work will constitute a part of the PhD theses of both authors.
Both authors are supported by the Max Planck Institute for Mathematics in Bonn.

\section{Recollections on K-theory of pushouts} \label{sec:K_theory_pushouts}

In general, $K$-theory (or more general localising invariants) does not preserve pushouts of stable categories.
This defect is studied in Land-Tamme's article \cite[Section 3]{LandTamme23}.
We summarise the results essential for our case, 
another exposition can also be found in \cite[Section 4]{BurklundLevy23}. 
Let us begin by recalling the notion of a localising invariant.

\begin{recollect}[Localising invariants]
    A \textit{Karoubi sequence} $\category{A} \xrightarrow{i} \category{B} \xrightarrow{p} \category{C}$ in $\catperf$ is a sequence in $\catperf$ together with a nullhomotopy $h \colon pi \simeq 0$ which exhibits the sequence as both a fiber and a cofiber sequence in $\catperf$, see \cite[Appendix A]{Hermitian2} for a discussion.
    For us, a \textit{localising invariant} (sometimes also called Karoubi localising invariant) is a functor $E \colon \catperf \to \category{E}$ with values in a stable category $\category{E}$ with the following two properties. 
    It satisfies $E(0) \simeq 0$ and for any Karoubi sequence $\category{A} \xrightarrow{i} \category{B} \xrightarrow{p} \category{C}$ in $\catperf$ the nullhomotopy $E(h) \colon E(p) E(i) \simeq E(0) \simeq 0$ exhibits the sequence $E(\category{A}) \to E(\category{B}) \to E(\category{C})$ as a fiber sequence in $\category{E}$.
    Equivalent characteristations of this notion can be found in \cite{HLS23}.
    Examples of localising invariants include nonconnective algebraic $K$-theory, topological Hochschild homology and topological cyclic homology, see \cite{BGT13} for a proof.
\end{recollect}

Essential for Land-Tamme's construction is the concept of \textit{partially lax pullbacks} as studied studied in \cite{Tamme18}. Let us recall their definition.
\begin{cons}[Partially lax pullback]
    For a cospan $\category{B} \xrightarrow{f} \category{D} \xleftarrow{g} \category{C}$ in $\cat$, the partially lax pullback $\category{B} \artimes \category{C}$ is the category with objects given by triples $(b, c, r)$ of objects $b \in \category{B}$, $c \in \category{C}$ and a morphism $r \colon f(b) \to g(c)$.
    Formally, it can be defined as the pullback
    \begin{equation*}
    \begin{tikzcd}
        \category{B} \artimes \category{C} \ar[r] \ar[d] \ar[dr, phantom, "\lrcorner", very near start] & \category{D}^{\Delta^1} \ar[d]
        \\ \category{B} \times \category{C} \ar[r, "{(f, g)}"] & \category{D} \times \category{D}.
    \end{tikzcd}
    \end{equation*}
    Mapping spaces in $\category{B} \artimes \category{C}$ are given by the pullback
    \begin{equation}\label[diagram]{diag:mapping_anima_lax_pullback}
    \begin{tikzcd}
        \Hom_{\category{B} \artimes \category{C}}((b, c, r),(b', c', r')) \ar[r] \ar[d] \ar[dr, phantom, "\lrcorner", very near start]
        & \Hom_{\category{C}}(c, c') \ar[d, "r^* \circ g"] \\
        \Hom_{\category{B}}(b, b') \ar[r, "r'_* \circ f"] & \Hom_{\category{D}}(f(b), g(c')).
    \end{tikzcd}
    \end{equation}
    Note that the pullback $\category{B} \times_{\category{D}} \category{C} \subseteq \category{B} \artimes \category{C}$ identifies with the full subcategory on objects $(b, c, r)$ for which $r \colon f(b) \to g(c)$ is an equivalence.
\end{cons}

\begin{obs}
    If $\category{B} \xrightarrow{f} \category{D} \xleftarrow{g} \category{C}$ is a diagram in $\catperf$, then the partially lax pullback $\category{B} \artimes \category{C}$ is again a perfect category.
    One essential property of the partially lax pullback is that $\category{B}$ and $\category{C}$, included via the maps $j_1 \colon \category{B} \to \category{B} \artimes \category{C}, b \mapsto (b,0,0)$ and $j_2 \colon \category{C} \to \category{B} \artimes \category{C}, c \mapsto (0,c,0)$, form a semiorthogonal decomposition of $\category{B} \artimes \category{C}$.
    In partiuclar, the projection $\category{B} \artimes \category{C} \to \category{B} \times \category{C}$ becomes an equivalence after applying any localizing invariant.
\end{obs}

Now consider a span $\category{B} \xleftarrow{b} \category{A} \xrightarrow{c} \category{C}$ in $\catperf$ and let $E$ be a localising invariant.
The natural map $E(\category{B}) \amalg_{E(\category{A})} E(\category{C}) \to E(\category{B} \amalg_{\category{A}} \category{C})$ not an equivalence in general.
This defect can be measured by the following construction.

\begin{cons}[{$\odot$}-product]\label{cons:odot_product}
    Consider a lax commutative diagram
    \begin{equation}\label[diagram]{diag:semicommutative_square_motives}
    \begin{tikzcd}
        \category{A} \ar[r, "c"] \ar[d, "b"'] 
        & \category{C} \ar[d, "q"]
        \\ 
        \category{B} \ar[r, "p"] \ar[ur, Rightarrow, "f"]
        & \category{D}
    \end{tikzcd}
    \end{equation}
    in $\catperf$, meaning that $f \colon pb \to qc$ is a natural transformation.
    The transformation $f$ induces a map $\category{A} \to \category{B} \artimes \category{C}$ and Land-Tamme define the $\odot$-product $\category{B} \odot_{\category{A}}^{\category{D}} \category{C}$ as the cofiber of this map in $\catperf$.
    The map $\category{B} \artimes \category{C} \to \category{B} \odot_{\category{A}}^{\category{D}} \category{C}$ is a Verdier localisation (see e.g. \cite[Theorem I.3.3]{NikolausScholze18} for an explanation of this notion) and we denote by $\im(\category{A})$ its fiber.
    Equivalently, it is the full perfect subcategory of $\category{B} \artimes \category{C}$ generated by the image of $\category{A}$.
    By the definition of localising invariants we get a pushout diagram
    \begin{equation}\label{diag:motivic_pushout}
    \begin{tikzcd}
        E(\im(\category{A})) \ar[r] \ar[d] \ar[dr, phantom, "\ulcorner", very near end] & E(\category{C}) \ar[d] \\
        E(\category{B}) \ar[r] & E(\category{B} \odot_{\category{A}}^{\category{D}} \category{C}).
    \end{tikzcd}
    \end{equation}
\end{cons}

The difficulty in working with these construction lies in the identification of the categories $\im(\category{A})$ and $\category{B} \odot_{\category{A}}^{\category{D}} \category{C}$.

\begin{cons}[{$\odot$-product associated to a pushout}]\label{cons:odot_for_pushout}
    To actually come back to our original problem about the $K$-theory of pushouts, consider a span $\category{B} \xleftarrow{b} \category{A} \xrightarrow{c} \category{C}$ in $\catperf$.
    Associated to it, we can construct a diagram of the shape of \cref{diag:semicommutative_square_motives} as follows
    \begin{equation}\label{diag:semicommutative_square_motiveic_pushout}
    \begin{tikzcd}
        \category{A} \ar[r, "c"] \ar[d, "b", swap] 
        & \category{C} \ar[d, "b^*c_*"]
        \\ 
        \category{B} \ar[ur, Rightarrow, "b^* \eta_c"] \ar[r] 
        & \ind(\category{B}),
    \end{tikzcd}
    \end{equation}
    where $\category{B} \to \ind(\category{B})$ is the Yoneda embedding, $c_* \colon \ind(\category{C}) \to \ind(\category{A})$ is a right adjoint to $c^* = \ind(c)$ and $\eta_c$ denotes the unit of this adjunction.
\end{cons}

Now the main technical theorem in \cite{LandTamme23} shows that the $\odot$-product is actually a pushout in this situation.
\begin{thm}[{\cite[Theorem 3.2]{LandTamme23}}] \label{thm:motivic_pushout}
    In the situation of \cref{cons:odot_for_pushout}, the square
    \begin{equation*}
    \begin{tikzcd}
        \category{A} \ar[r, "c"] \ar[d, "b", swap] 
        & \category{C} \ar[d, "j_2"]
        \\ 
        \category{B} \ar[r, "\Omega j_1"] 
        & \category{B} \odot_{\category{A}}^{\ind(\category{B})} \category{C}
    \end{tikzcd}
    \end{equation*}
    is a pushout in $\catperf$.
\end{thm}

\section{The main theorem}

We will now turn to our main subsect of interest: For an exact functor $\alpha \colon \category{C} \rightarrow \category{C}$ we study the category of twisted automorphisms $\category{C}_{h\bbN}$ and prove that $E(\category{C}_{h\bbN})$ naturally splits for every localising invariant $E$. As useful companions, we will study categories of twisted endomorphisms and the notion of nilpotence for twisted endomorphisms.

\subsection{Twisted endomorphisms and automorphisms}\label{sec:twisted_automorphisms}

This subsection contains the construction of twisted endomorphism and automorphism categories and some of their basic properties, which will be useful in applications.
Let us begin by recalling the construction of lax equalisers, whose basic properties can be found in \cite[Proposition II.1.5]{NikolausScholze18}.
\begin{cons}[Lax equaliser]
    Given two functors $f, g \colon \category{C} \to \category{D}$, the lax equaliser $\laxeq(f, g)$ consists of pairs $(x, r)$ of an object $x \in \category{C}$ and a map $r \colon f(x) \to g(x)$ in $\category{D}$.
    Formally, it is defined as the pullback
    \begin{equation}\label{diag:pullback_lax_equlaiser}
    \begin{tikzcd}
        \laxeq(f,g) \ar[r] \ar[d] \ar[dr, phantom, "\lrcorner", very near start] & \category{D}^{\Delta^1} \ar[d]
        \\ \category{C} \ar[r, "{(f, g)}"] & \category{D} \times \category{D}
    \end{tikzcd}
    \end{equation}
    from which one obtains the formula for mapping spaces as the equaliser
    \begin{equation}\label{eq:mapping_anima_lax_equaliser}
        \Hom_{\laxeq(f,g)}((x,r), (y, s)) \simeq \equaliser\left(s_* f, r^* g \colon \Hom_{\category{C}}(x, y) \to \Hom_{\category{D}}(f(x), g(y))\right).
    \end{equation}
\end{cons}

From now on, consider a category $\category{C}$ together with an endofunctor $\alpha \colon \category{C} \xrightarrow{} \category{C}$.

\begin{defn}[Twisted endormophisms and automorphisms]\label{def:twisted_endomorphism}
    We define the category $\End_\alpha(\category{C})$ of \textit{twisted endomorphisms} as the lax equaliser $\End_\alpha(\category{C}) = \laxeq(\id, \alpha \colon \category{C} \to \category{C})$.
    Objects consist of pairs $(x, f)$ of objects $x \in \category{C}$ and maps $f \colon x \to \alpha(x)$ in $\category{C}$.
    Similarly, define $\Endbar_\alpha(\category{C}) = \laxeq(\alpha, \id \colon \category{C} \to \category{C})$ with objects pairs $(x,f)$ of objects $x \in \category{C}$ and maps $f \colon \alpha(x) \to x$ in $\category{C}$.
    
    We also define the category $\aut_\alpha(\category{C})$ of \textit{twisted automorphisms} as the equaliser $\aut_\alpha(\category{C}) = \equaliser(\id, \alpha \colon \category{C} \to \category{C})$.
    Objects consist of pairs $(x, f)$ of objects $x \in \category{C}$ and equivalences $f \colon x \xrightarrow{\simeq} \alpha(x)$ in $\category{C}$.
\end{defn}

We will relate these notions to categories of modules over tensor algebras and twisted polynomial rings in \cref{cons:tensor_algebras}.

\begin{obs}
    The map $\category{D} \xrightarrow{\id_(-)} \category{D}^{\Delta^1} $ induces a natural map $\aut_\alpha(\category{C}) \subseteq \End_\alpha(\category{C})$, identifying the source as the full subcategory on all pairs $(x,f)$ for which $f \colon x \to \alpha(x)$ is an equivalence.

    Denote by $\fgt \colon \End_\alpha(\category{C}) \to \category{C}, (x,f) \mapsto x$ the forgetful funtor.
    It is conservative by \cite[Proposition II.1.5 (ii)]{NikolausScholze18}.
    If $\category{C}$ is a perfect (resp. presentable) category and $\alpha$ is exact (resp. a left adjoint, resp. a right adjoint), then $\End_\alpha(\category{C})$ and $\aut_\alpha(\category{C})$ are perfect (resp. presentable) categories and the functors $\aut_\alpha(\category{C}) \to \End_\alpha(\category{C})$ and $\fgt \colon \End_\alpha(\category{C}) \to \category{C}$ are exact (resp. a left adjoint, resp. a right adjoint).
    This directly follows from the fact that the forgetful functors $\catperf, \lpresentable, \rpresentable \to \cat$ preserve limits.
\end{obs}

\begin{constr}[Twisted powers]
    \label{constr:twisted_powers}
    The functos $\alpha$ induces an endofunctor $\overline{\alpha} \colon \End_\alpha(\category{C}) \to \End_\alpha(\category{C}), (x, f) \mapsto (\alpha x, \alpha f)$, which formally can be constructed using the functoriality of lax equalisers:
    There is a natural transformation $\shifttransformation{1} \colon \id \to \overline{\alpha}$ of endofunctors of $\End_\alpha(\category{C})$ given on objects by the commutative square
    \begin{equation*}
    \begin{tikzcd}
        x \ar[r, "f"] \ar[d, "f"] 
        & \alpha x \ar[d, "\alpha f"] \\
        \alpha x \ar[r, "\alpha f"]
        & \alpha^2 f,
    \end{tikzcd}
    \end{equation*}
    see e.g. \cite[Construction II.5.2]{NikolausScholze18} for a formal construction.
    We can define the $n$-fold composite $[n] = ({\overline{\alpha}^{n-1}}[1]) \circ ({\overline{\alpha}^{n-2}}[1]) \dots \circ [1] \colon \id \to \overline{\alpha}^n$ of $[1]$.
    We denote the image of the map $[n] \colon (x,f) \to \overline{\alpha}^n(x,f)$ in $\category{C}$ by
    \begin{equation*}
        \twistedpower{f}{n} \coloneqq \left( x \xrightarrow{f} \alpha(x) \xrightarrow{\alpha(f)} \dots \xrightarrow{\alpha^{n-1}(f)} \alpha^n(f) \right).
    \end{equation*}
\end{constr}

\begin{constr}[Localisation and free twisted endomorphisms]\label{constr:localisation}\label{constr:free_endomorphism}
    Suppose that $\category{C}$ is presentable and that $\alpha \colon \category{C} \to \category{C}$ is a right adjoint.
    Then the inclusion $\aut_\alpha(\category{C}) \subseteq \End_\alpha(\category{C})$ admits a left adjoint $\loc{\alpha} \colon \End_\alpha(\category{C}) \to \aut_\alpha(\category{C})$.
    It is explicitly given by the formula
    \begin{equation*}
        \loc{\alpha} = \colim \left( \id \xrightarrow{[1]} \overline{\alpha} \xrightarrow{\overline{\alpha} [1]} \overline{\alpha}^2 \xrightarrow{\overline{\alpha}^2 [1]} \dots \right)
    \end{equation*}
    as shown in the proof of \cite[Proposition II.5.3]{NikolausScholze18}.
    Similarly, the forgetful functor $\fgt \colon \End_\alpha(\category{C}) \to \category{C}$ admits a left adjoint $\freeend{\alpha} \colon \category{C} \to \End_\alpha(\category{C})$, the \textit{free twisted endomorphism}, and the forgetful functor $\fgt \Endbar_\alpha(\category{C}) \to \category{C}$ admits a left adjoint $\freeendbar{\alpha} \colon \category{C} \to \Endbar_\alpha(\category{C})$
\end{constr}
As an illustration of the previous construction the reader should keep induction along the ring homomorphisms $R[t] \to R[t, t^{-1}]$ and $R \to R[t]$ in mind.
An explicit description of $\freeend{\alpha}$ is as follows.
\begin{prop}\label{prop:free_endomorphism}
    Let $\category{C}$ be presentable and suppose that $\alpha \colon \category{C} \to \category{C}$ admits a left adjoint $\alpha^L$.
    Then $\freeend{\alpha}(x) = (\coprod_{n \ge 0} (\alpha^L)^n(x), \shift)$, where $\shift$ is given on the $n$-th summand by the adjunction unit $\eta \colon (\alpha^L)^n(x) \to \alpha \alpha^L (\alpha^L)^n(x) = \alpha (\alpha^L)^{n+1}(x)$.
    
    If $\alpha$ additionally preserves countable coproducts, then $\freeend{\alpha}(x) = (\coprod_{n \ge 0} \alpha^n(x), \shift)$, where $\shift \colon \alpha(\coprod_{n \ge 0} \alpha^n(x)) \simeq \coprod_{n \ge 1} \alpha^n(x) \to \coprod_{n \ge 0} \alpha^n(x)$ is the inclusion.
\end{prop}
\begin{proof}
    We only prove the formula for $\freeend{\alpha}$, the proof of the formula for $\freeendbar{\alpha}$ being similar.
    Denote $F(x) = (\coprod_{n \ge 0} (\alpha^L)^n(x), \shift)$.
    The inclusion $x \to \coprod_{n \ge 0} (\alpha^L)^n(x)$ of the zeroth summand induces by adjunction a map $\freeend{\alpha}(x) \to F(x)$.
    We claim that this is an equivalence.
    Denote by $i \colon x \to \fgt \freeend{\alpha}(x)$ the map adjoint to the identity on $\freeend{\alpha}(x)$.
    The maps 
    \begin{equation*}
        (\alpha^L)^n (x) \xrightarrow{i} (\alpha^L)^n \fgt \freeend{\alpha}(x) \xrightarrow{[n]} (\alpha^L)^n \alpha^n \fgt \freeend{\alpha}(x) \xrightarrow{\varepsilon^n} \fgt \freeend{\alpha}(x)
    \end{equation*}
    assemble into a map $\coprod_{n \ge 0} (\alpha^L)^n(x) \to \fgt \freeend{\alpha}(x)$.
    We want to enhance this to a map $a \colon F(x) \to \freeend{\alpha}(x)$.
    Note that we have the commutative diagram
    \begin{equation*}\footnotesize
    \begin{tikzcd}
        (\alpha^L)^n (x) \ar[r, "{i}"] \ar[dd, "\eta"]
        & (\alpha^L)^n \fgt \freeend{\alpha}(x) \ar[r, "{[n]}"] \ar[dd, "\eta"]
        & (\alpha^L)^n \alpha^n \fgt \freeend{\alpha}(x) \ar[r, "\varepsilon^n"] \ar[d, "{[1]}"]
        & \fgt \freeend{\alpha}(x) \ar[dd, "{[1]}"] 
        \\
        &
        & (\alpha^L)^{n} \alpha^{n+1} \fgt \freeend{\alpha}(x) \ar[dr, "\varepsilon^n"]\ar[d, "\eta"]
        &
        \\
        \alpha (\alpha^L)^{n+1} (x) \ar[r, "{i}"]
        & \alpha (\alpha^L)^{n+1} \fgt \freeend{\alpha}(x) \ar[r, "{[n+1]}"]
        & \alpha (\alpha^L)^{n+1} \alpha^{n+1} \fgt \freeend{\alpha}(x) \ar[r, "\varepsilon^{n+1}"'] 
        & \alpha \fgt \freeend{\alpha}(x),
    \end{tikzcd}
    \end{equation*}
    where the bottom right triangle commutes via the triangle identities.
    The outer commutative rectangle provides a map $F(x) \to \freeend{\alpha}(x)$.
    We show these two maps are inverse to each other.
    By construction, the composite $x \to \fgt F(x) \to \fgt \freeend{\alpha}(x)$ is $i$ which, by adjunction, shows that the composite $\freeend{\alpha}(x) \to F(x) \to \freeend{\alpha}(x)$ is the identiy.
    We can show that the composite $F(x) \to \freeend{\alpha(x)} \to F(x)$ is an equivalence after applying the conservative functor $\fgt$.
    Note that in the commutative diagram
    \begin{equation*}
    \begin{tikzcd}
        (\alpha^L)^n x \ar[r, "i"]
        & (\alpha^L)^n \fgt \freeend{\alpha}(x) \ar[r, "{[n]}"] \ar[d,"a"]
        & (\alpha^L)^n \alpha^n \fgt \freeend{\alpha}(x) \ar[r, "\varepsilon^n"] \ar[d, "a"]
        & \fgt \freeend{\alpha}(x) \ar[d, "a"]
        \\
        & (\alpha^L)^n \fgt F(x) \ar[r, "{[n]}"]
        & (\alpha^L)^n \alpha^n \fgt F(x) \ar[r, "\varepsilon^n"]
        & \fgt F(x)
    \end{tikzcd}
    \end{equation*}
    the composite $(\alpha^L)^n x \to \fgt F(x)$ taking the bottom route is equivalent to the inclusion of the $n$-th summand in $\fgt F(x)$.
    This shows that the map $\fgt F(x) \to \fgt \freeend{\alpha}(x) \to \fgt F(x)$ is equivalent to the identity.
\end{proof}

Recall that if $\category{C}$ is a presentable stable $\infty$-category, then $\category{C}^{\omega}$, its subcategory of compact objects, is perfect.

\begin{lem}\label{lem:twisted_endomorphisms_compactly_generated}
    Suppose that $\category{C}$ is a compactly generated stable presentable category and that $\alpha \colon \category{C} \to \category{C}$ preserves limits and colimits.
    Then $\End_\alpha(\category{C})$ and $\aut_\alpha(\category{C})$ are compactly generated.
    The image of $\category{C}^\omega$ under $\freeend{\alpha} \colon \category{C} \to \End_\alpha(\category{C})$ and $\loc{\alpha} \circ \freeend{\alpha} \colon \category{C} \to \aut_\alpha(\category{C})$ are families of compact generators.
\end{lem}
\begin{proof}
    By construction, $\freeend{\alpha}$ admits a colimit preserving right adjoint and thus preserves compact objects.
    To show that the image of $\category{C}^\omega$ under $\freeend{\alpha}$ generates $\End_\alpha(\category{C})$ it suffices to show that if $(y,f) \in \End_\alpha(\category{C})$ such that $\Hom_{\End_\alpha(\category{C})}(\freeend{\alpha}(x), (y,f)) \simeq 0$ for all $x \in \category{C}^\omega$, then $(y,f) \simeq 0$.
    But as $\Hom_{\End_\alpha(\category{C})}(\freeend{\alpha}(x), (y,f)) \simeq \Hom_{\category{C}}(x,y)$ this follows from $\category{C}$ being compactly generated.
    The argument for $\aut_\alpha$ is analogous.
\end{proof}

Finally, let us define the notion of $\bbN$-orbits $(-)_{h\bbN}$ appearing in \cref{thm:twisted_BHS} and relate them to twisted automorphisms.

\begin{obs}[(Co)limits over $B \bbN$]\label{obs:colimit_over_BN}
    Let $\category{D}$ be a category with finite colimits and consider an object $x \in \category{D}$ together with an endomorphism $f \colon x \to x$.
    Denote by $B\bbN$ the category with a single object and $\bbN$ as its endomorpism monoid.
    The endomorphism $f$ corresponds to a diagram $F \colon B\bbN \to \category{D}$ with $F(*) = x$ and $F(1) = f$.
    The colimit $x_{h \bbN} \coloneqq \colim_{B\bbN} F$ in $\category{D}$ then fits into the pushout
    \begin{equation}\label{diag:colimits_BN}
    \begin{tikzcd}
        x \amalg x \ar[r, "\id \amalg \id"] \ar[d, "\id \amalg f"] \ar[dr, phantom, very near end, "\ulcorner"]
        & x \ar[d]
        \\
        x \ar[r] 
        & x_{h \bbN}.
    \end{tikzcd}
    \end{equation}
    One way to see this is to represent $B \bbN$ as the pushout of the span $* \leftarrow * \amalg * \rightarrow \Delta^1$ and to apply \cite[Corollary 1.3]{HorevYanovski17}.
    This has a dual version if $\category{D}$ has finite limits.
    The limit of the diagram $F$ is then given by the pullback
    \begin{equation}\label{diag:limits_BN}
    \begin{tikzcd}
        x^{h\bbN} \ar[r] \ar[d] \ar[dr, phantom, very near start, "\lrcorner"]
        & x \ar[d, "{(\id, \alpha)}"]
        \\
        x \ar[r, "{(\id, \id)}"] 
        & x \times x.
    \end{tikzcd}
    \end{equation}
\end{obs}

\begin{lem}\label{lem:identification_mapping_torus_twisted_automorphisms}
    Let $\alpha \colon \category{C} \to \category{C}$ be an exact endofunctor of a perfect category.
    Then there is an equivalence $\category{C}_{h \bbN} \simeq \left(\aut_{\alpha^R}(\ind(\category{C}))\right)^\omega$, where $\alpha^R \colon \ind(\category{C}) \to \ind(\category{C})$ denotes the right adjoint to $\ind(\alpha)$.
\end{lem}
\begin{proof}
    Recall that colimits in $\catperf$ are computed by taking compact objects in the colimit of the $\ind$-completed diagram in $\lpresentable$, or equivalently in the limit of the right adjoint diagram in $\rpresentable$.
    We obtain
    \begin{equation*}
        \category{C}_{h \bbN} \simeq \left( \ind(\category{C})_{h \bbN} \right)^\omega \simeq \left( \ind(\category{C})^{h \bbN} \right)^\omega
    \end{equation*}
    where the limit in the last step is formed over the diagram $B\bbN \to \cat$ classified by $\alpha^R \colon \ind(\category{C}) \to \ind(\category{C})$.
    \cref{obs:colimit_over_BN} identifies $\ind(\category{C})^{h \bbN} \simeq \aut_{\alpha^R}(\ind(\category{C}))$.
\end{proof}

\subsection{Twisted nilpotent endomorphisms}\label{sec:nilpotent_endomorphisms}

In this subsection we define twisted nilpotent endomorphisms and prove some of their alternative characterisations.
Let $\alpha \colon \category{C} \to \category{C}$ is an exact endomorphism of a perfect category.
There are maps
\begin{equation*}
    \triv \colon \category{C} \to \End_\alpha(\category{C}), x \mapsto (x, 0 \colon x \to \alpha(x)) \quad \text{and} \quad \trivbar \colon \category{C} \to \Endbar_\alpha(\category{C}), x \mapsto (x, 0)
\end{equation*}
sending an object to the zero endomorphism on that object.

\begin{defn}\label{def:nilpotent_endomorphism}
    We define the category $\Nil_\alpha(\category{C})$ of \textit{twisted nilpotent endomorphisms} as the full perfect subcategory of $\End_\alpha(\category{C})$ generated by the image of $\triv \colon \category{C} \to \End_\alpha(\category{C})$.
    Similarly, we define the category $\Nilbar_\alpha(\category{C})$ as the full perfect subcategory of $\Endbar_\alpha(\category{C})$ generated by the image of $\trivbar \colon \category{C} \to \Endbar_\alpha(\category{C})$.
\end{defn}

Next we investigate how objects in $\Nil_{\alpha}(\category{C})$ are related to actual nilpotent endomorphisms.
Note that the condition on $\alpha$ in the following is in particular satsified if $\alpha$ is an equivalence.

\begin{prop}\label{prop:characterisation_nilpotent_endomorphism}
    Suppose that $\alpha$ has a left adjoint $\alpha^L$.
    Then following categories agree:
    \begin{enumerate}[label=(\arabic*)]
        \item $\Nil_\alpha(\category{C})$;
        \item The full subcategory of $\End_\alpha(\category{C})$ on those objects $(x,f)$ for which there exists some $n \in \bbN$ with $\twistedpower{f}{n} \simeq 0$ (see \cref{constr:twisted_powers} for this notation);
        \item the kernel $\ker(\loc{\alpha} \colon \End_\alpha(\ind(\category{C}))^\omega \to \aut_\alpha(\ind(\category{C}))^\omega)$.
    \end{enumerate}
\end{prop}
\begin{proof}
    For the inclusion $(2) \subseteq (3)$, consider $(x,f) \in \End_\alpha(\category{C})$ and $n \in \bbN$ such that $\twistedpower{f}{n} \simeq 0$. 
    Using cofinality of $n\bbN \subset \bbN$ as posets and the formula for $\loc{\alpha}$ from \cref{constr:localisation}, we obtain
    \begin{align*}
        \loc{\alpha} (x,f) 
        & = \colim \left( (x,f) \xrightarrow{f} \alpha(x,f) \xrightarrow{\alpha f} \dots \right) \\
        & \simeq \colim \left( (x,f) \xrightarrow{\twistedpower{f}{n}} \alpha^n(x,f) \xrightarrow{\alpha^n \twistedpower{f}{n}} \dots \right) 
        \simeq 0.
    \end{align*}
    
    For $(1) \subseteq (3)$, note that the category described in $(3)$ is a perfect subcategory of $\End_\alpha(\category{C})^\omega$ which contains the image of $\triv$ by the argument above.
    Thus it also contains $\Nil_\alpha(\category{C})$.
    
    For the inclusion $(3) \subseteq (1)$, we first claim that for $(x,f) \in \End_\alpha(\ind(\category{C}))^\omega$ and $k \ge 1$, the object $\fib(\shifttransformation{k} \colon (x,f) \to \overline{\alpha}^n(x,f))$ lies in $\Nil_\alpha(\category{C})$.
    For $k = 1$, recall from \cref{lem:twisted_endomorphisms_compactly_generated} that $\End_\alpha(\ind(\category{C}))^\omega$ is generated by the elements $\freeend{\alpha}(z)$ for $z \in \category{C}$.
    As the claim is stable under retracts, shifts and fiber sequences, it suffices to consider $(x,f)$ of this form.
    By the explicit formula $\freeend{\alpha}(z) = (\bigoplus_{n \ge 0} (\alpha^L)^n(z), \shift)$ from \cref{prop:free_endomorphism} and the fact that $\alpha \colon \ind(\category{C}) \to \ind(\category{C})$ preserves colimits, it is easy to see that 
    \begin{equation*}
        \fib([1] \colon \freeend{\alpha}(z) \to \overline{\alpha}\freeend{\alpha}(z)) \simeq \triv(\Omega \alpha (z)).
    \end{equation*}
    The case of general $k$ follows by induction using the fiber sequence $\fib(\shifttransformation{k-1}) \to \fib(\shifttransformation{k}) \to \fib(\overline{\alpha}^{k-1}\shifttransformation{1})$ coming from $\shifttransformation{k} \simeq \overline{\alpha}^{k-1} \shifttransformation{1} \circ \shifttransformation{k-1}$.
    
    We can now show $(3) \subseteq (1)$.
    Let $(x,f) \in \ker(\loc{\alpha} \colon \End_\alpha(\ind(\category{C}))^\omega \to \aut_\alpha(\ind(\category{C}))^\omega)$.
    By compactness of $(x,f)$ and the formula for $\loc{\alpha}$ from \cref{constr:localisation} we obtain
    \begin{align*}
        0 &\simeq \Hom_{\aut_\alpha(\ind(\category{C}))}(\loc{\alpha} (x,f), \loc{\alpha} (x,f)) \simeq \Hom_{\End_\alpha(\ind(\category{C}))}((x,f), \loc{\alpha} (x,f)) 
        \\ &\simeq \colim\left( \Hom_{\End_\alpha(\ind(\category{C}))}((x,f), (x,f)) \xrightarrow{[1]} \Hom_{\End_\alpha(\ind(\category{C}))}((x,f), \overline{\alpha} (x,f)) \xrightarrow{\overline{\alpha}[1]} \dots \right).
    \end{align*}
    In particular, $\id_{(x,f)}$ vanishes in a finite stage of the colimit showing that for large $n$ we have $0 \simeq \shifttransformation{n} \colon (x,f) \to \overline{\alpha}^n(x,f)$.
    Its fiber, which lies in (1) by the previous claim, is given by $(x,f) \oplus \overline{\alpha}^n(x,f)$ and contains $(x,f)$ as a retract showing $(x,f) \in \Nil_\alpha(\category{C})$.
    
    This also proves $(3) \subseteq (2)$: We just saw that for $(x,f)$ in the kernel in $(3)$ and large $n$ we have with $0 \simeq [n] \colon (x,f) \to \overline{\alpha}^n(x,f)$, which reduces to $0 \simeq \twistedpower{f}{n} \colon x \to \alpha^n(x)$  on underlying objects.
    It remains to argue that $x \in \category{C}$.
    But this follows from $(3) \subseteq (1)$ as it implies $(x,f) \in \Nil_\alpha(\category{C}) \subseteq \End_\alpha(\category{C})$.
\end{proof}

\subsection{\texorpdfstring{$K$-theory of twisted automorphisms}{K-theory of twisted automorphisms}}\label{sec:proof_decomposition}

In this subsection we will prove the main result of this article, \cref{thmintro:bhs_decomposition}, about the splitting of $E(\category{C}_{h\bbN})$ for an exact endofunctor $\alpha \colon \category{C} \rightarrow \category{C}$ of a perfect category and a localising invariant $E$.
By \cref{obs:colimit_over_BN} we have the pushout
\begin{equation*}
\begin{tikzcd}
    \category{C} \oplus \category{C} \ar[r, "\id \oplus \id"] \ar[d, "\id \oplus \alpha"] \ar[dr, phantom, very near end, "\ulcorner"]
    & \category{C} \ar[d]\\
    \category{C} \ar[r] 
    & \category{C}_{h \bbN}
\end{tikzcd}
\end{equation*}
in $\catperf$.
Applying Land-Tamme's theory of $K$-theory of pushouts, more precisely \cref{cons:odot_product,cons:odot_for_pushout,thm:motivic_pushout}, we obtain the commutative square
\begin{equation}\label{diag:motivic_pullback_imcc}
\begin{tikzcd}
    \im(\category{C} \oplus \category{C}) \ar[r, "\id \oplus \id"] \ar[d, "\id \oplus \alpha"] & \category{C} \ar[d]\\
    \category{C} \ar[r] & \category{C}_{h \bbN}
\end{tikzcd}
\end{equation}
which becomes a pushout after applying any localising invariant.
For the proof of \cref{thm:twisted_BHS} it remains to understand $E(\im(\category{C} \oplus \category{C}))$.
We will show that it admits an orthogonal decomposition by the categories $\Nil_{\alpha}(\category{C})$ and $\Nilbar_{\alpha}(\category{C})$ from \cref{def:nilpotent_endomorphism}.

Note that in the given situation, the right adjoint to $\id \oplus \id \colon \category{C} \oplus \category{C} \to \category{C}$ already exists before $\ind$-completion and is given by the diagonal functor $\Delta \colon \category{C} \to \category{C} \oplus \category{C}$.
The corresponding square in \cref{diag:semicommutative_square_motiveic_pushout} is thus given by
\begin{equation*}
\begin{tikzcd}
    \category{C} \oplus \category{C} \ar[r, "\id \oplus \id"] \ar[d, "\id \oplus \alpha", swap] 
    & \category{C} \ar[d, "\id + \alpha"]
    \\
    \category{C} \ar[r, "\id", swap] \ar[ur, Rightarrow, "f"]
    & \category{C},
\end{tikzcd}
\end{equation*}
where the homotopy $f$ at an object $(x,y) \in \category{C} \oplus \category{C}$ is the transformation 
\begin{equation*}
    f = \id_x + \id_{\alpha(y)} \colon x \oplus \alpha y \to x \oplus y \oplus \alpha x \oplus \alpha y.
\end{equation*}
The partially lax pullback $\category{C} \artimes \category{C}$ then has objects given by triples $(x,y,r)$ consisting of objects $x,y \in \category{C}$ and a map $r \colon x \to y \oplus \alpha(y)$.
The induced map $(i_1, i_2) \colon \category{C} \oplus \category{C} \to \category{C} \artimes \category{C}$ is given by 
\begin{align}
    &i_1(x) =(x, x, (\id, 0) \colon x \to x \oplus \alpha(x)) \quad \text{and} \quad \label{eq:inclusion_into_lax_pullback} \\
    &i_2(y) = (\alpha(y), y, (0, \id) \colon \alpha(y) \to y \oplus \alpha(y)). \nonumber
\end{align}
There are also maps
\begin{align*}
    &j_1 \colon \End_{\alpha}(\category{C}) \to \category{C} \artimes \category{C}, (x, r \colon x \to \alpha(x)) \mapsto (x, x, (\id, r) \colon x \to x \oplus \alpha(x)), \\
    &j_2 \colon \Endbar_{\alpha}(\category{C}) \to \category{C} \artimes \category{C}, \ (y, s \colon \alpha(y) \to y) \mapsto (\alpha(y), y, (s, \id) \colon \alpha(y) \to y \oplus \alpha(y)).
\end{align*}

\begin{lem}
    The functors $j_1$ and $j_2$ are full inclusions which identify $\End_{\alpha}(\category{C})$ (and $\Endbar_{\alpha}(\category{C})$) with the full subcategory of $\category{C} \artimes \category{C}$ on objects $(x, y, r)$ for which the first (resp. second) component of $r \colon x \to y \oplus \alpha(y)$ is an equivalence.
\end{lem}
\begin{proof}
    We only prove the statement about $j_1$ with the other case being analogous.
    The pullback square in \cref{diag:mapping_anima_lax_pullback},
\begin{equation*}
\begin{tikzcd}
    \Hom_{\category{C} \artimes \category{C}}(j_1(x,f), j_1(y,g)) \ar[r] \ar[d] \ar[dr, phantom, very near start, "\lrcorner"]
    & \Hom_{\category{C}}(x,y) \ar[d, "{(\id, g_*)}"]
    \\
    \Hom_{\category{C}}(x,y) \ar[r, "{(\id, f^* \alpha)}"]
    & \Hom_{\category{C}}(x,y) \times \Hom_{\category{C}}(x, \alpha(y)),
\end{tikzcd}
\end{equation*}
identifies with
\begin{equation*}
    \equaliser(f^* \alpha, g_* \colon \Hom(x,y) \to \Hom(x, \alpha(y))) \simeq \Hom_{\End_{\alpha}(\category{C})}((x,f), (y,g)).\qedhere
\end{equation*}
\end{proof}

Let us briefly recall the notion of a semiorthogonal decomposition.

\begin{recollect}[Semiorthogonal decomposition]\label{rec:semiorthogonal_decomposition}
    Recall that an semiorthogonal decomposition of a perfect category $\category{D}$ consists of two full perfect subcategories $\category{D}_0, \category{D}_1 \subseteq \category{D}$ such that $\category{D}_0 \cup \category{D}_1$ generates $\category{D}$ as a perfect category and $\Hom_{\category{D}}(x_1, x_0) \simeq 0$ for all $x_i \in \category{D}_i$.
    In this situation, the inclusion $\category{D}_1 \subseteq \category{D}$ admits a right adjoint $p_1$ and the inclusion $\category{D}_1 \subseteq \category{D}$ admits a left adjoint $p_0$.
    Importantly, the sequence $\category{D}_0 \hookrightarrow \category{D} \xrightarrow{p_1} \category{D}_1$ is a Karoubi sequence.
    As $\category{D}_0 \hookrightarrow \category{D}$ admits the retraction $p_1$, we even see that for any localising invariant $E$ the map $E(\category{D}_0) \oplus E(\category{D}_1) \to E(\category{D})$ induced by the inclusions is an equivalence.
    More information on (semi)orthogonal decompositions can be found in \cite[Section II.7.2]{SAG} or \cite[Section A.8]{HA}.
    There is also the notion of an orthogonal decomposition where one additionally requires that $\Hom_{\category{D}}(x_0, x_1) \simeq 0$
\end{recollect}

\begin{lem}\label{lem:nil_orthogonal_decomposition}
    The subcategories $\Nil_\alpha(\category{C})$ and $\Nilbar_{\alpha}(\category{C})$ (included via $j_1$ and $j_2$) form an orthogonal decomposition of $\im(\category{C} \oplus \category{C})$.
\end{lem}

\begin{proof}
    By definition, $\im(\category{C} \oplus \category{C})$ is the perfect subcategory of $\category{C} \artimes \category{C}$ generated by the images of the functors $i_1, i_2 \colon \category{C} \to \category{C} \artimes \category{C}$ from \cref{eq:inclusion_into_lax_pullback}.
    Notice that $i_1$ factors as the composite $\category{C} \xrightarrow{\triv} \Nil_\alpha(\category{C}) \lhook\joinrel\xrightarrow{j_1} \category{C} \artimes \category{C}$.
    Similarly, $i_2$ factors through $\Nilbar_\alpha(\category{C})$.
    Thus, $\Nil_\alpha(\category{C})$ and $\Nilbar_{\alpha}(\category{C})$ generate $\im(\category{C} \oplus \category{C})$

    To prove orthogonality, note that the vanishing of $\Hom_{\category{D}}(x_0, x_1)$ is stable under colimits and retracts in the first variable and under limits and retracts in the second variable, so it suffices to check it for the generating sets given by the image of $i_1$ and $i_2$.
    The pullback \cref{diag:mapping_anima_lax_pullback} for mapping spaces in partially lax pullbacks specialises for objects $x, y \in \category{C}$ to the pullback
   \begin{equation}
    \begin{tikzcd}
        \Hom_{\category{C} \artimes \category{C}}(i_1(x), i_2(y)) \ar[r] \ar[d] \ar[dr, phantom, very near start, "\lrcorner"]
        & \Hom_{\category{C}}(x, y) \ar[d, "i_1"] \\
        \Hom_{\category{C}}(x, \alpha y) \ar[r, "i_2"] 
        & \Hom_{\category{C}}(x, y \oplus \alpha y)
    \end{tikzcd}
    \end{equation}
    with the lower and right maps induced by the inclusions of the summands of $y \oplus \alpha(y)$.
    But as the pullback of the cospan $y \rightarrow y \oplus \alpha(y) \leftarrow \alpha(y)$ is trivial, this shows that $\Hom_{\category{C} \artimes \category{C}}(i_1(x), i_2(y)) \simeq 0$.
    Analogously one shows $\Hom_{\category{C} \artimes \category{C}}(i_2(y), i_1(x)) \simeq 0$.
\end{proof}

Let us come back to determining $E(\category{C}_{h \bbN})$.
Note that the functor $\triv \colon \category{C} \to \Nil_\alpha(\category{C})$ admits a retraction given by the forgetful functor $\fgt \colon \Nil_\alpha(\category{C}) \to \category{C}$.
Defining the $\Nil$-term as $NE_{\alpha}(\category{C}) \coloneqq \Sigma \cofib(\triv \colon E(\category{C}) \to E(\Nil_\alpha(\category{C}))$,
the retraction $\fgt$ induces a splitting 
\begin{equation}\label{eq:splitting_E_Nil}
    E(\Nil_\alpha(\category{C})) \simeq E(\category{C}) \oplus \Omega NE_{\alpha}(\category{C}). 
\end{equation}
Similarly, the functor $\trivbar \colon \category{C} \to \Nilbar_\alpha(\category{C})$ admits a retraction $\fgt \colon \Nilbar_\alpha(\category{C}) \to \category{C}$ which induces a splitting $E(\Nilbar_\alpha(\category{C})) \simeq E(\category{C}) \oplus \Omega \NKbar{E}_{\alpha}(\category{C})$.

\begin{thm}\label{thm:twisted_BHS}
    The assembly map $E(\category{C})_{h \bbN} \to E(\category{C}_{h \bbN})$ splits and there is a natural equivalence
    \begin{equation*}
        E(\category{C}_{h \bbN}) \simeq E(\category{C})_{h \bbN} \oplus NE_{\alpha}(\category{C}) \oplus \NKbar{E}_{\alpha}(\category{C}).
    \end{equation*}
\end{thm}
\begin{proof}
    From \cref{thm:motivic_pushout} we obtain a fiber sequence
    \begin{equation}\label{eq:fiber_sequence_BHS_motivic_pushout}
        E(\im(\category{C} \oplus \category{C})) \to E(\category{C} \artimes \category{C}) \to E(\category{C}_{h \bbN}).
    \end{equation}
    Recall that the two projections $\pr_i \colon \category{C} \artimes \category{C} \to \category{C}$ induce a splitting 
    \begin{equation*}
        \pr_1 \oplus \pr_2 \colon E(\category{C} \artimes \category{C}) \xrightarrow{\simeq} E(\category{C}) \oplus E(\category{C}).
    \end{equation*}
    Furthermore, by the orthogonal decomposition in \cref{lem:nil_orthogonal_decomposition}, the inclusions induce an equivalence $E(\Nil_{\alpha}(\category{C})) \oplus E(\Nil_{\alpha^{-1}}(\category{C})) \xrightarrow{\simeq} E(\im(\category{C} \oplus \category{C}))$ so the fiber sequence \cref{eq:fiber_sequence_BHS_motivic_pushout} reduces to the fiber sequence
    \begin{equation}
        E(\Nil_{\alpha}(\category{C})) \oplus E(\Nilbar_{\alpha}(\category{C})) \to E(\category{C}) \oplus E(\category{C}) \to E(\category{C}_{h \bbN}).
    \end{equation}
    Note that the composite $\Nil_{\alpha}(\category{C}) \xrightarrow{j_1} \category{C} \artimes \category{C} \xrightarrow{\pr_i} \category{C}$ is given by $\fgt$ for $i = 1,2$ and the composite $\Nilbar_{\alpha}(\category{C}) \xrightarrow{j_2} \category{C} \artimes \category{C} \xrightarrow{\pr_i} \category{C}$ is given by $\alpha \fgt$ for $i = 1$ and $\fgt$ for $i = 2$.
    Using the splitting \cref{eq:splitting_E_Nil} we obtain the fiber sequence
    \begin{equation*}
        E(\category{C}) \oplus E(\category{C}) \oplus \Omega NE_{\alpha}(\category{C}) \oplus \Omega NE_{\alpha}(\category{C}) 
        \xrightarrow{(\id, \id) \oplus (\alpha, \id) \oplus 0 \oplus 0} E(\category{C}) \oplus E(\category{C}) 
        \to E(\category{C}_{h \bbN}).
    \end{equation*}
    Note that it contains the fiber sequence
    \begin{equation*}
         E(\category{C}) \oplus E(\category{C})
        \xrightarrow{(\id, \id) \oplus (\alpha, \id)} 
        E(\category{C}) \oplus E(\category{C}) 
        \rightarrow E(\category{C})_{h \bbN}
    \end{equation*}
    as a retract, which proves the claimed splitting.
\end{proof}

\subsection{\texorpdfstring{$K$-theory of twisted endomorphisms}{K-theory of twisted endomorphisms}}\label{sec:k_theory_twisted_endomorphisms}

In this section we will explain how the $NE$-terms appearing in the splitting in \cref{thm:twisted_BHS} are related to $E(\End_{\alpha}(\ind(\category{C}))^\omega)$.
Land-Tamme showed a variant of this for the $E$-theory of tensor algebras in \cite[Corollary 4.5, Proposition 4.7]{LandTamme23}.
Their proof relies on the computation of certain endomorphism rings using their work on theory of $K$-theory of pullbacks \cite{LandTamme19}.
We give a more direct argument by working on the categorical level.
This also extends \cref{thm:splitting_ktheory_endomorphisms} to categories not generated by a single element.

Let us again consider an exact endofunctor $\alpha \colon \category{C} \to \category{C}$ of a perfect category and denote by $\alpha^R \colon \ind(\category{C}) \to \ind(\category{C})$ the right adjoint to $\ind(\alpha)$.
By abuse of notation, we will often not distinguish between $\ind(\alpha)$ and $\alpha$.
\begin{thm}\label{thm:splitting_ktheory_endomorphisms}
    The map $\freeend{\alpha^R} \colon \category{C} \to \End_{\alpha^R}(\ind(\category{C}))^\omega$ induces a splitting
    \begin{equation*}
        E(\End_{\alpha^R}(\ind(\category{C}))^\omega) \simeq E(\category{C}) \oplus NE_\alpha(\category{C}).
    \end{equation*}
    Similarly, the map $\freeendbar{\alpha^R} \colon \category{C} \to \Endbar_{\alpha^R}(\ind(\category{C}))^\omega$ induces a splitting
    \begin{equation*}
        E(\Endbar_{\alpha^R}(\ind(\category{C}))^\omega) \simeq E(\category{C}) \oplus \NKbar{E}_\alpha(\category{C}).
    \end{equation*}
\end{thm}
Note that the main result in \cref{thm:splitting_ktheory_endomorphisms} is not the existence of a splitting for $E(\End_{\alpha^R}(\ind(\category{C}))^\omega)$, but rather that (up to a shift) the error terms are the same as those appearing in the splitting for $E(\Nil_\alpha(\category{C}))$.

We will only prove the splitting for $E(\End_{\alpha^R}(\ind(\category{C}))^\omega)$, the argument for the splitting of $E(\Endbar_{\alpha^R}(\ind(\category{C}))^\omega)$ being analogous.
Our approach is similar to the proof of \cref{thm:twisted_BHS} by first presenting $\Endbar_{\alpha^R}(\ind(\category{C}))^\omega$ as a pushout in $\catperf$ and then applying Land-Tamme's machinery from \cref{sec:K_theory_pushouts} to it.
Let us begin by recalling the tensoring construction for perfect categories.
\begin{recollect}[Tensoring of $\catperf$ over $\cat$]
    Given a category $I$, the functor $\func(I, -) \colon \catperf \to \catperf$ admits a left adjoint $- \otimes I$, the tensoring of $\catperf$ over $\cat$.
    It is the initial perfect category together with a functor $\category{C} \times I \to \category{C} \otimes I$ which is exact in the first variable and has an explicit description given by $\category{C} \otimes I = \func(I\op, \ind(\category{C}))^\omega$.
    Under this identification, for an object $i \in I$ the inclusion $a_i \colon \category{C} \times \{i\} \to \category{C} \otimes I$ is left adjoint to evaluation $\eval_i \colon \func(I\op, \ind(\category{C})) \to \ind(\category{C})$.
    More details on the non idempotent complete version of this construction can be found in \cite[Remark 6.4.2]{Hermitian1} or \cite[Section 2]{Saunier22}.
\end{recollect}
We are only be interested in the case $I = \Delta^1$ so that $\category{C} \otimes \Delta^1 \simeq \func((\Delta^1)\op, \ind(\category{C}))^\omega$.
The left adjoints to the evaluation functors are explicitly given by $a_0(x) = (x \leftarrow 0)$ and $a_1(x) = (x \xleftarrow{\id} x)$, where we use leftwards pointing arrows to indicate that we are working in $(\Delta^1)\op$.
We can now give a presentation of $\End_{\alpha^R}(\ind(\category{C}))^\omega$ as a pushout.
\begin{lem}
    There is a pushout square in $\catperf$ of the form
    \begin{equation}\label{diag:pushout_twisted_endomorphisms}
    \begin{tikzcd}
        \category{C} \oplus \category{C} \ar[r, "{a_0 \oplus a_1}"] \ar[d, "{\alpha \oplus \id}"'] \ar[dr, phantom, very near end, "\ulcorner"]
        & \category{C}\otimes\Delta^1 \ar[d]
        \\
        \category{C} \ar[r, "\freeend{\alpha^R}"]
        & \End_{\alpha^R}(\ind(\category{C}))^\omega.
    \end{tikzcd}
    \end{equation}
\end{lem}
\begin{proof}
    To calculate the pushout of the upper left span in \cref{diag:pushout_twisted_endomorphisms} one instead passes to right adjoints after $\ind$ completion and forms the pullback.
    The upper left span then becomes the cospan 
    \begin{equation*}
        \ind(\category{C}) \xrightarrow{(\alpha^R, \id)} \ind(\category{C}) \times \ind(\category{C}) \xleftarrow{(\eval_0, \eval_1)} \func((\Delta^1)\op, \ind(\category{C})).
    \end{equation*}
    After the identification $\Delta^1 \simeq (\Delta^1)\op$ this becomes the usual pullback square from \cref{diag:pullback_lax_equlaiser} defining the lax equaliser 
    \begin{equation*}
    \begin{tikzcd}
        \End_{\alpha^R}(\ind(\category{C})) \ar[r] \ar[d] \ar[dr, phantom, very near start, "\lrcorner"]
        &\ind(\category{C})^{\Delta^1} \ar[d, "{(\eval_0, \eval_1)}"] \\
        \ind(\category{C}) \ar[r, "{(\id, \alpha^R)}"]
        & \ind(\category{C}) \times \ind(\category{C}).
    \end{tikzcd}
    \end{equation*}
    The left vertical arrow here is given by $\fgt$, so its left adjoint $\freeend{\alpha^R}$ appears in the diagram in the statement.
\end{proof}
Applying Land-Tamme's theory of $K$-theory of pushouts, more precisely \cref{cons:odot_product,cons:odot_for_pushout,thm:motivic_pushout}, we obtain the commutative square
\begin{equation}\label{diag:motivic_pushout_twisted_endomorphisms}
\begin{tikzcd}
    \im(\category{C} \oplus \category{C}) \ar[r, "i_0 \oplus i_1"] \ar[d, "\alpha \oplus \id"] 
    & \category{C}\otimes\Delta^1 \ar[d]
    \\
    \category{C} \ar[r]
    & \End_{\alpha^R}(\ind(\category{C}))^\omega
\end{tikzcd}
\end{equation}
which becomes an equivalence after applying any localising invariant.
We again use the notation $\im(\category{C} \oplus \category{C})$ even though it is different from the category denoted by the same symbol in \cref{sec:proof_decomposition}.
As before, the main difficulty is the description of this category.
The lax square in \cref{diag:semicommutative_square_motiveic_pushout} associated to the pushout \cref{diag:pushout_twisted_endomorphisms} is given by
\begin{equation*}
\begin{tikzcd}
    \category{C} \oplus \category{C} \ar[r, "{a_0 \oplus a_1}"] \ar[d, "{\alpha \oplus \id}"']
    & \category{C}\otimes\Delta^1 \ar[d, "\alpha \eval_0 \oplus \eval_1"]
    \\
    \category{C} \ar[r, "\id"'] \ar[ur, Rightarrow, "f"]
    & \category{C}.
\end{tikzcd}
\end{equation*}
Note here that the right adjoint $(\eval_0, \eval_1)$ to $a_0 \oplus a_1$ preserves compact objects as it sends the generators $a_0(x)$ and $a_1(x)$ of $\category{C}\otimes\Delta^1$ to compact objects in $\ind(\category{C})$.
The transformation $f$ is given on $(x,y)$ by the map $\id_{\alpha(x)} + \id_y \colon \alpha(x) \oplus y \to \alpha(x) \oplus \alpha(y) \oplus y$.
The corresponding partially lax pullback $\category{C} \artimes (\category{C}\otimes\Delta^1)$ then has objects $(x, y \xleftarrow{f} z, r \colon x \to \alpha(y) \oplus z)$ with $x \in \category{C}$ and $f \in \category{C}\otimes\Delta^1$.
The induced functor $(i_1, i_2) \colon \category{C} \oplus \category{C} \to \category{C} \artimes (\category{C}\otimes\Delta^1)$ is given by
\begin{align*}
    &i_1(x) = (\alpha(x), a_0(x), (\id, 0) \colon \alpha(x) \to \alpha(x) \oplus 0), \\
    &i_2(x) = (x, a_1(x) , (0, \id) \colon x \to \alpha(x) \oplus x).
\end{align*}
There is also a functor
\begin{equation*}
    j \colon \End_{\alpha}(\category{C}) \to \category{C} \artimes (\category{C}\otimes\Delta^1), \hspace{3mm} (x, f \colon x \to \alpha(x)) \mapsto (x, a_1(x), (f, \id) \colon x \to \alpha(x) \oplus x).
\end{equation*}
\begin{lem}
   The functors  $i_1$ and $j$ are fully faithful.
\end{lem}
\begin{proof}
    From \cref{diag:mapping_anima_lax_pullback} we have the pullback square
    \begin{equation*}
    \begin{tikzcd}
        \Hom_{\category{C} \artimes (\category{C}\otimes\Delta^1)}(i_1(x), i_1(y)) \ar[r] \ar[d] \ar[dr, phantom, very near start, "\lrcorner"]
        & \Hom_{\category{C}\otimes\Delta^1}(a_0(x), a_0(y)) \ar[d, "\alpha \eval_0"]
        \\
        \Hom_{\category{C}}(\alpha(x),\alpha(y)) \ar[r, "\id"]
        & \Hom_{\category{C}}(\alpha(x),\alpha(y)).
    \end{tikzcd}
    \end{equation*}
    Note that $\eval_0 \colon \Hom_{\category{C}\otimes\Delta^1}(a_0(x), a_0(y)) \to \Hom_{\category{C}}(x, y)$ is an equivalence from which the claim about $i_1$ follows.
    Similarly, we have the pullback square
    \begin{equation*}
    \begin{tikzcd}
        \Hom_{\category{C} \artimes (\category{C}\otimes\Delta^1)}(j(x,f), j(y,g)) \ar[r] \ar[d] \ar[dr, phantom, very near start, "\lrcorner"]
        & \Hom_{\category{C}\otimes\Delta^1}(a_1(x), a_1(y)) \ar[d, "(f^* \alpha \eval_0) + \eval_1"]
        \\
        \Hom_{\category{C}}(x, y) \ar[r, "g_* + \id"]
        & \Hom_{\category{C}}(x, \alpha(y) \oplus y).
    \end{tikzcd}
    \end{equation*}
    
    As $\eval_0, \eval_1 \colon \Hom_{\category{C}\otimes\Delta^1}(a_1(x), a_1(y)) \to \Hom_{\category{C}}(x,y)$ are equivalences, this pullback identifies with $\equaliser(f^* \alpha, g_* \colon \Hom_{\category{C}}(x, y) \to \Hom_{\category{C}}(x, \alpha(y)) \simeq \Hom_{\End_{\alpha}(\category{C})}((x,f), (y,g))$.
\end{proof}

\begin{lem}\label{lem:semiorthogonal_decomposition_twisted_endomorphism}
    The subcategories $\Nil_\alpha(\category{C})$ and $\category{C}$ (included via $j$ and $i_1$) form a semiorthogonal decomposition of $\im(\category{C} \oplus \category{C})$.
\end{lem}
\begin{proof}
    It is clear that the perfect subcategory generated by $\Nil_\alpha(\category{C})$ and $\category{C}$ contains $\im(\category{C} \oplus \category{C})$ as $i_2 \simeq j \circ \triv$.
    But $\Nil_\alpha(\category{C})$ is also contained in $\im(\category{C} \oplus \category{C})$ as it is generated by $\category{C}$.
    Together, this shows that $\Nil_\alpha(\category{C})$ and $\category{C}$ generate $\im(\category{C} \oplus \category{C})$.
    For semiorthogonality, \cref{diag:mapping_anima_lax_pullback} gives us the pullback
    \begin{equation*}
    \begin{tikzcd}
        \Hom_{\category{C} \artimes (\category{C}\otimes\Delta^1)}(j(x,f), i_1(y)) \ar[r] \ar[d] \ar[dr, phantom, very near start, "\lrcorner"]
        & \Hom_{\category{C}\otimes\Delta^1}(a_1(x), a_0(y)) \ar[d, "f_* \alpha \eval_0"]
        \\
        \Hom_{\category{C}}(x, \alpha(y)) \ar[r, "\id"]
        & \Hom_{\category{C}}(x, \alpha(y)).
    \end{tikzcd}
    \end{equation*}
    We have $\Hom_{\category{C}\otimes\Delta^1}(a_1(x), a_0(y)) \simeq \Hom_{\category{C}}(x, 0) \simeq 0$, using that $a_1$ is left adjoint to $\eval_1$, which proves the desired vanishing.
\end{proof}
As a final step towards the identification of $\Nil$-terms, we need to show a splitting for the $K$-theory of $\category{C} \otimes \Delta^1$.
\begin{lem}\label{lem:semiorthogonal_decomposition_tensor}
    The functors $a_0, a_1 \colon \category{C} \hookrightarrow \category{C} \otimes \Delta^1$ form a semiorthogonal decomposition of $\category{C} \otimes \Delta^1$
\end{lem}
\begin{proof}
    Recall that $a_0$ and $a_1$ are left adjoint to $\eval_0$ and $\eval_1$.
    From this it is easy to see that $a_0$ and $a_1$ are fully faithful.
    They generate $\category{C} \otimes \Delta^1$ as their images are compact generators of the stable presentable category $\func((\Delta^1)\op, \ind(\category{C}))$, again using that they are left adjoint to $\eval$.
    Semiorthogonality follows from $\Hom_{\category{C}\otimes\Delta^1}(a_1(x), a_0(y)) \simeq \Hom_{\category{C}}(x, 0) \simeq 0$.
\end{proof}

\begin{proof}[Proof of \cref{thm:splitting_ktheory_endomorphisms}.]
    By \cref{thm:motivic_pushout}, applying $E$ to the commutative square \cref{diag:motivic_pushout_twisted_endomorphisms} gives the pushout
    \begin{equation*}
    \begin{tikzcd}
        E(\im(\category{C} \oplus \category{C})) \ar[r, "i_0 \oplus i_1"] \ar[d, "\alpha \oplus \id"] \ar[dr, phantom, very near end, "\ulcorner"]
        & E(\category{C}\otimes\Delta^1) \ar[d]
        \\
        E(\category{C}) \ar[r, "\freeend{\alpha^R}"]
        & E(\End_{\alpha^R}(\ind(\category{C}))^\omega).
    \end{tikzcd}
    \end{equation*}
    Splitting the top terms using the semiorthogonal decompositions from \cref{lem:semiorthogonal_decomposition_twisted_endomorphism,lem:semiorthogonal_decomposition_tensor} together with the splitting $E(\Nil_\alpha(\category{C})) \simeq E(\category{C}) \oplus \Omega NE_\alpha(\category{C})$ from \cref{eq:splitting_E_Nil}, we arrive at the pushout square
    \begin{equation}\label{diag:splitting_diagram_k_theory_end}
    \begin{tikzcd}
        E(\category{C}) \oplus \Omega NE_\alpha(\category{C}) \oplus E(\category{C}) \ar[r, "{(\id, 0) \oplus 0 \oplus (0, \id)}"] \ar[d, "\id \oplus 0 \oplus \id"'] \ar[dr, phantom, very near end, "\ulcorner"]
        & E(\category{C}) \oplus E(\category{C}) \ar[d]
        \\
        E(\category{C}) \ar[r, "\freeend{\alpha^R}"]
        & E(\End_{\alpha^R}(\ind(\category{C}))^\omega).
    \end{tikzcd}
    \end{equation}
    The component $E(\category{C}) \oplus E(\category{C}) \to E(\category{C}) \oplus E(\category{C})$ of the upper horizontal map is the identity.
    To check this, note that by construction the composite $\Nil_\alpha(\category{C}) \xrightarrow{j} \category{C} \artimes (\category{C}\otimes\Delta^1) \to \category{C}\otimes\Delta^1$ is equivalent to $a_1 \fgt$ and the composite $\category{C} \xrightarrow{i_1} \category{C} \artimes (\category{C}\otimes\Delta^1) \to \category{C}\otimes\Delta^1$ is equialent to $a_0$.
    Furthermore, the cofibers of the upper and lower horizontal maps in the pushout square \cref{diag:splitting_diagram_k_theory_end} are equivalent, from which we obtain 
    \begin{equation*}
        \cofib(\freeend{\alpha^R} \colon E(\category{C}) \to E(\End_{\alpha^R}(\ind(\category{C}))^\omega) \simeq NE_\alpha(\category{C}).
    \end{equation*}
    As the final step in proving the claimed splitting, note that $\freeend{\alpha}$ has a retraction, sending a twisted endomorphism $f \colon x \to \alpha^R(x)$ to the fiber $\fib(f^L \colon \alpha(x) \to x)$ of the adjoint twisted endomorphism.
    To be precise, this only defines a functor $\fib \colon \End_{\alpha^R}(\ind(\category{C})) \to \ind(\category{C})$.
    Using the formula for $\freeend{\alpha^R}$ in \cref{prop:free_endomorphism}, we see that $\fib$ is a retraction of $\freeend{\alpha^R}$ on the large level.
    But $\fib$ sends the generators $\freeend{\alpha^R}(x)$ for $x \in \category{C}$ of $\End_{\alpha^R}(\ind(\category{C}))^\omega$ to compact objects and thus preserves all compact objects.
    
    An analogous argument, using the pushout 
    \begin{equation*}
    \begin{tikzcd}
        \category{C} \oplus \category{C} \ar[r, "{a_0 \oplus a_1}"] \ar[d, "{\id \oplus \alpha}"'] \ar[dr, phantom, very near end, "\ulcorner"]
        & \category{C}\otimes\Delta^1 \ar[d]
        \\
        \category{C} \ar[r, "\freeendbar{\alpha^R}"]
        & \Endbar_{\alpha^R}(\ind(\category{C}))^\omega
    \end{tikzcd}
    \end{equation*}
    instead of \cref{diag:pushout_twisted_endomorphisms}, proves 
    \begin{equation*}
        \cofib( \freeendbar{\alpha^R} \colon E(\category{C}) \to E(\Endbar_{\alpha^R}(\ind(\category{C}))^\omega) \simeq \NKbar{E}_\alpha(\category{C}).
    \end{equation*}
    To obtain the retraction of $\freeendbar{\alpha^R}$ in the final step, one has to note that the right adjoint $\alpha^R \colon \ind(\category{C}) \to \ind(\category{C})$ preserves colimits as its left adjoint $\ind(\alpha)$ preserves compact objects.
    Then one can apply \cref{prop:free_endomorphism} for an explicit formula for $\freeendbar{\alpha^R}$ to see that $\cofib \colon \Endbar_{\alpha^R}(\ind(\category{C}))^\omega \to \category{C}, (x, f \colon \alpha(x) \to x) \mapsto \cofib(f)$ is a retraction of $\freeendbar{\alpha^R}$.
\end{proof}

\section{Regularity}\label{sec:regularity}

In this section we want to generalise the classical vanishing result for $\Nil$-terms saying that $NK_\alpha(R) \simeq 0$ if $R$ is a regular ring and $\alpha \colon R \to R$ an automorphism, as shown for example in \cite[Theorem 4]{Waldhausen78}.
When passing from rings to (derived) categories of modules, the analogue of regularity is the notion of a $t$-structure.
Let us recall its definition.
\begin{recollect}[$t$-structure]\label{recollect:t_structure}
    A $t$-structure on a stable $\category{C}$ consists of two full subcategories $\category{C}_{\le 0}, \category{C}_{\ge 0} \subseteq \category{C}$ satisfying the following properties:
    \begin{enumerate}[label=(\arabic*)]
        \item for all $x \in \category{C}_{\ge 0}$ and $y \in \category{C}_{\le 0}$ one has $\Hom_{\category{C}}(x,\Omega y) \simeq 0$;
        \item $\category{C}_{\ge 0}$ is closed under $\Sigma$ and $\category{C}_{\le 0}$ is closed under $\Omega$;
        \item for any $x \in \category{C}$ there is a fiber sequence $x' \to x \to x''$ with $x' \in \category{C}_{\ge 0}$ and $x'' \in \Omega \category{C}_{\le 0}$.
    \end{enumerate}
    In this situation, the inclusion $\category{C}_{\le 0} \subseteq \category{C}$ has a left adjoint $\tau_{\le 0}$ and the inclusion $\category{C}_{\ge 0} \subseteq \category{C}$ has a right adjoint $\tau_{\ge 0}$.
    A $t$-structure is bounded if for any $x \in \category{C}$ there is $k \in \bbN$ with $\Sigma^k x \in \category{C}_{\ge 0}$ and  $\Omega^k x \in \category{C}_{\le 0}$.
    The heart of $\category{C}$ is defined by $\category{C}^\heartsuit = \category{C}_{\ge 0} \cap \category{C}_{\le 0}$.
    More details on $t$-structures on stable categories can be found in \cite[Section 1.2.1]{HA}.
    
    As an important class of examples of bounded $t$-structures, consider a regular coherent discrete ring $R$.
    This means that any finitely generated left $R$-module is finitely presented.
    In this situation, the perfect derived category $\module_R^\omega$ admits a bounded $t$-structure with (co)connective objects given by precisely those $R$-modules with homology in nonnegative (resp. nonpositive) degrees.
    For more details and examples of $t$-structures on module categories of ring spectra we refer the reader to \cite{BurklundLevy22}.
\end{recollect}

Let us now state the general vanishing result for $\Nil$-terms.
It can be easily deduced Burklund-Levy's abstract dévissage result \cite[Theorem 1.3]{BurklundLevy23}.
Land-Tamme show the analogous result for tensor algebras in the case where $\category{C}$ has a single generator in \cite[Corollary 4.13]{LandTamme23}.

\begin{cor}\label{thm:regularity_vanishing_nil}
    Consider an exact endofunctor $\alpha \colon \category{C} \to \category{C}$ of a perfect category and assume that $\category{C}$ admits a bounded $t$-structure.
    If $\alpha$ is left t-exact, meaning that $\alpha(\category{C}_{\le 0}) \subseteq \category{C}_{\le 0}$, then $\tau_{\ge 0} NK_\alpha(\category{C}) \simeq 0$.
    If the heart $\category{C}^\heartsuit$ is aditionally Noetherian, then $NK_\alpha(\category{C}) \simeq 0$.
    The analougous vanishing results hold for $\NKbar{K}_\alpha(\category{C})$ if $\alpha$ is right $t$-exact.
\end{cor}
\begin{proof}
    We want to apply \cite[Theorem 1.3]{BurklundLevy23} to the functor $\triv \colon \category{C} \to \Nil_\alpha(\category{C})$.
    By definition, the image of $\triv$ generates $\Nil_\alpha(\category{C})$ under finite colimits and retracts.
    It remains to check that the restriction of $\triv$ to $\category{C}^\heartsuit$ is fully faithful.
    For $x,y \in \category{C}^\heartsuit$ one has $\Hom_{\End_\alpha(\category{C})}((x,0), (y,0)) \simeq \Hom_{\category{C}}(x,y) \times \Omega \Hom_{\category{C}}(x, \alpha(y)) \simeq \Hom_{\category{C}}(x,y)$, where the last step we uses that $\alpha$ is left $t$-exact.
\end{proof}

The $t$-structure constructed on $\Nil_\alpha(\category{C})$ in the proof of \cite[Theorem 1.3]{BurklundLevy23} is very inexplicit.
We can give a more direct construction of a $t$-structure on $\End_{\alpha}(\category{C})$ and $\Nil_{\alpha}(\category{C})$ which might be of independent interest.

\begin{prop}\label{prop:t_structure_endomorphisms}
    Suppose that $\category{C} \in \catperf$ admits a $t$-structure and that $\alpha \colon \category{C} \to \category{C}$ is left t-exact.
    Then the full subcategories $\End_\alpha(\category{C})_{\ge 0}$ (resp. $\End_\alpha(\category{C})_{\le 0}$) on those objects $(x,f)$ with $x \in \category{C}_{\ge 0}$ (resp. $x \in \category{C}_{\le 0}$) define a $t$-structure on $\End_\alpha(\category{C})$, called the \textit{pointwise $t$-structure}.
\end{prop}
\begin{proof}
    We have to verify the properties from \cref{recollect:t_structure}.
    Recall that for objects $(x,f), (y,g) \in \End_\alpha(\category{C})$ we have
    \begin{equation*}
        \Hom_{\End_\alpha(\category{C})}((x,f), (y,g)) 
        \simeq \equaliser (f^* \circ \alpha,g_* \colon \Hom_{\category{C}}(x,y) \to \Hom_{\category{C}}(x,\alpha(y))).
    \end{equation*}
    Now if $x \in \category{C}_{\ge 0}$ and $y \in \category{C}_{\le -1}$, then $\alpha(y) \in \category{C}_{\le -1}$ showing that $\Hom_{\category{C}}(x,y) \simeq \Hom_{\category{C}}(x,\alpha(y)) \simeq 0$.
    This also implies $\Hom_{\End_\alpha(\category{C})}((x,f), (y,g)) \simeq 0$ and verifies condition (1).
    As $\Sigma(x,f) \simeq (\Sigma x, \Sigma f)$ it follows that $\End_\alpha(\category{C})_{\ge 0}$ is closed under $\Sigma$ and $\End_\alpha(\category{C})_{\le 0}$ is closed under $\Omega$ which shows (2).
    
    Finally, for (3) we have to show that every object $(x,f) \in \End_\alpha(\category{C})$ sits in a fiber sequence with a 1-connective and 0-coconnective object.
    As $\alpha$ is only left $t$-exact, it generally does not commute with the truncation $\tau_{\le 0}$.
    However, there is always a Beck-Chevalley transformation $\beta \colon \tau_{\le 0} \alpha \to \alpha \tau_{\le 0}$ associated to the commutative square
    \begin{equation*}
    \begin{tikzcd}
        \category{C}_{\le 0} \ar[r, hook] \ar[d, "\alpha"]
        & \category{C} \ar[d, "\alpha"]
        \\
        \category{C}_{\le 0} \ar[r, hook]
        & \category{C},
    \end{tikzcd}
    \end{equation*}
    where the horizontal left adjoints are precisely the truncations $\tau_{\le 0}$.
    Denoting by $\eta \colon \id \to \tau_{\le 0}$ the adjunction unit, we obtain the commutative diagram
    \begin{equation}\label{diag:t_structure_end}
    \begin{tikzcd}
        x \ar[d, "f"] \ar[r, "\eta_x"]
        & \tau_{\le 0} x \ar[d, "\tau_{\le 0} f"] 
        \ar[dr, bend left]
        \\
        \alpha(x) \ar[r, "\eta_{\alpha(x)}"]
        & \tau_{\le 0} \alpha(x) \ar[r, "\beta"]
        & \alpha(\tau_{\le 0} x).
    \end{tikzcd}
    \end{equation}
    It follows from \cite[Lemma 2.2.4(3)]{CSY22} that $\beta \circ \eta_{\alpha(x)} \simeq \alpha \eta_x$.
    The outer quadrilateral of \cref{diag:t_structure_end} defines a map $(x,f) \to (\tau_{\le 0} x, \beta \circ \tau_{\le 0} f)$ in $\End_\alpha(\category{C})$ with underlying map $\eta_x \colon x \to \tau_{\le 0}x$.
    The underlying object of the fiber of $(x,f) \to (\tau_{\le 0} x, \beta \circ \tau_{\le 0} f)$ is $\tau_{\ge 1} x \simeq \fib(x \to \tau_{\le 0} x)$ and thus 1-connective.
\end{proof}

Let us now show that the pointwise $t$-structure on $\End_\alpha(\category{C})$ restricts to a $t$-structure on $\Nil_\alpha(\category{C})$ in many situations.
Note that the condition of the following result is in particular satisfied if $\alpha$ is a left $t$-exact automorphism of $\category{C}$.

\begin{cor}
    Suppose that $\category{C} \in \catperf$ admits a bounded $t$-structure and that $\alpha \colon \category{C} \to \category{C}$ is left t-exact and admits a left adjoint.
    Then the pointwise $t$-structure on $\End_\alpha(\category{C})$ restricts to a $t$-structure on $\Nil_\alpha(\category{C})$.
\end{cor}
\begin{proof}
    We only have to show that $\Nil_\alpha(\category{C})$ is closed under truncation.
    Recall from \cref{prop:characterisation_nilpotent_endomorphism} that $\Nil_\alpha(\category{C})$ is the full subcategory of $\End_\alpha(\category{C})$ consisting of those objects $(x,f)$ for which $\twistedpower{f}{n} \simeq 0$ for large $n$.
    In the proof of \cref{prop:t_structure_endomorphisms} we saw the explicit formula for truncation $\tau_{\le 0}(x,f) = (\tau_{\le 0} x, \beta \circ \tau_{\le 0} f)$, where $\beta \colon \tau_{\le 0} \alpha \to \alpha \tau_{\le 0}$ is the Beck-Chevalley transformation.
    Now note that by naturality of $\beta$ the composite
    \begin{equation*}
        \tau_{\le 0} x 
        \xrightarrow{\tau_{\le 0} f} \tau_{\le 0} \alpha(x) 
        \xrightarrow{\beta} \alpha \tau_{\le 0} x 
        \xrightarrow{\alpha \tau_{\le 0} f} \dots 
        \xrightarrow{\alpha^{n-1} \tau_{\le 0} f} \alpha^{n-1} \tau_{\le 0} \alpha (x) 
        \xrightarrow{\alpha^{n-1} \beta} \alpha^{n} \tau_{\le 0} x
    \end{equation*}
    defining $\twistedpower{(\beta \circ \tau_{\le 0} f)}{n}$ is equivalent to the composite
    \begin{equation*}
    \tau_{\le 0} x 
    \xrightarrow{\tau_{\le 0} f} \tau_{\le 0} \alpha(x) 
    \xrightarrow{\tau_{\le 0} \alpha(f)} \dots 
    \xrightarrow{\alpha^{n-2} \beta} \alpha^{n-1} \tau_{\le 0} \alpha (x) 
    \xrightarrow{\alpha^{n-1} \beta} \alpha^{n} \tau_{\le 0} x
    \end{equation*}
    given by $\twistedpower{\beta}{n} \circ \tau_{\le 0} \twistedpower{f}{n}$ which vanishes for large $n$.
    This shows $\tau_{\le 0}(x,f) \in \Nil_\alpha(\category{C})$.
\end{proof}

\section{Applications and Examples}\label{sec:applications}

\subsection{\texorpdfstring{$K$-theory of tensor algebras}{K-theory of tensor algebras}}

In this section we will study various applications of \cref{thm:twisted_BHS} to obtain splittings for the $K$-theory of certain rings.
Let us begin with a very general consideration.
\begin{constr}[Tensor algebras]\label{cons:tensor_algebras}
    For a ring spectrum $R \in \alg(\spectra)$ denote by $\module_R$ the category of left $R$-modules (in spectra).
    Consider an endomorphism $\alpha \colon \module_R^\omega \to \module_R^\omega$.
    Note that, equivalently, $\alpha \simeq M \otimes_R -$ for a $(R,R)$-bimodule $M$ which is compact as a left $R$-module by \cite[Remark 4.8.4.9]{HA}.
    Denote by $\alpha^R \colon \module_R \to \module_R$ the $\ind$-right adjoint to $\alpha$.
    As explained in \cref{lem:twisted_endomorphisms_compactly_generated}, the categories $\End_{\alpha^R}(\module_R)$ and $\aut_{\alpha^R}(\module_R)$ are compactly generated by the elements $\freeend{\alpha^R}(R)$ and $\loc{\alpha^R}\freeend{\alpha^R}(R)$, respectively.
    Lurie's version of the Schwede-Shipley theorem \cite[Theorem 7.1.2.1, Remark 7.1.2.3]{HA} shows that $\End_{\alpha^R}(\module_R) \simeq \module_{\tensoralg{R}{M}}(\spectra)$ and $\aut_{\alpha^R}(\module_R) \simeq \module_{\tensoralgaut{R}{M}}(\spectra)$, where
    \begin{align*}
        &\tensoralg{R}{M} \coloneqq \Endspec_{\End_{\alpha^R}(\module_R)}(\freeend{\alpha^R}(R)) 
        \quad \text{and} \quad \\
        &\tensoralgaut{R}{M} \coloneqq \Endspec_{\aut_{\alpha^R}(\module_R)}(\loc{\alpha}\freeend{\alpha^R}(R))
    \end{align*}
    are the endomorphism ring spectra.
    We call them the \textit{tensor algebra} and \textit{localised tensor algebra}.
    Using the description of free endomorpisms in \cref{prop:free_endomorphism}, we can caluclate the underlying spectrum of the tensor algebra by
    \begin{equation*}
        \Endspec_{\End_{\alpha^R}(\module_R)}(\freeend{\alpha^R}(R)) \simeq \hom_{\module_R}(R, \fgt \freeend{\alpha^R}(R)) \simeq \bigoplus_{n \ge 0} M^{\otimes_R^n},
    \end{equation*}
    where the first equivalence is given by restriction along the unit $R \to \fgt \freeend{\alpha^R}(R)$.
    The inverse sends a map $f \colon R \to \bigoplus_{n \ge 0} M^{\otimes_R^n}$ to the map
    \begin{equation*}
        \bigoplus_{k \ge 0} M^{\otimes_R^k} \otimes_R f \colon \bigoplus_{k \ge 0} M^{\otimes_R^k} \to \bigoplus_{l \ge 0} M^{\otimes_R^l}.
    \end{equation*}
    We thus see that multiplication on $\tensoralg{R}{M}$ is componentwise given by the map
    \begin{equation*}
        M^{\otimes_R^n} \times M^{\otimes_R^m} \to M^{\otimes_R^n} \otimes_R M^{\otimes_R^m} \simeq M^{\otimes_R^{n+m}}.
    \end{equation*}
    Let us now turn to the localised tensor algebra.
    \cref{constr:localisation} gives us the formula
    \begin{align*}
        &\tensoralgaut{R}{M} 
        \simeq \loc{\alpha}(\tensoralg{R}{M}) \\
        &\simeq \colim\left(\bigoplus_{k \ge 0} M^{\otimes_R^k} \xrightarrow{\coev} \bigoplus_{k \ge 0} M^\vee \otimes_R M^{\otimes_R^k} \xrightarrow{\coev} \bigoplus_{k \ge 0} (M^\vee)^{\otimes_R^2} \otimes_R M^{\otimes_R^k} \xrightarrow{\coev} \dots \right),
    \end{align*}
    where $M^\vee$ denotes the right dual $(R,R)$-bimodule to $M$ (such that $M \otimes_R -$ is left adjoint to $M^\vee \otimes_R -$), which exists as $M$ is compact as a left $R$-module.
    The maps in the colimit are induced by the coevaluation $\coev \colon R \to M^\vee \otimes_R M$.
\end{constr}
\begin{example}
    Let $R \in \alg(\spectra)$ be a ring spectrum together with an endomorphism $\alpha \colon R \to R$.
    Induction induces an endomorphism $\alpha_! \colon \module_R^\omega \to \module_R^\omega$.
    In that case, the bimodule $M$ from \cref{cons:tensor_algebras} is given by $M = R$ with trivial left $R$-module structre and right $R$-module structure given by $\alpha$.
    We write $R_\alpha[t] \coloneqq \tensoralg{R}{M}$ and obtain as a left $R$-module
    \begin{equation*}
        R_\alpha[t]  \simeq \fgt \freeend{\alpha^R}(R) \simeq \bigoplus_{n \ge 0} \alpha_!^n R \simeq \bigoplus_{n \ge 0} R.
    \end{equation*}
    The multiplication $\alpha_!^n R \times \alpha_!^m R \to \alpha_!^{n+m} R$ identifies with $R \times R \xrightarrow{\id \otimes \alpha^n} R \times R \xrightarrow{\otimes} R$.
    In particular, if the ring $R$ is discrete, this recovers the classical twisted polynomial ring.
    
    If $\alpha$ is an automorphism, the module $M^\vee$ is given by the $(R, R)$-bimodule $R$ with trivial right $R$-module structure and left $R$-module structure given by $\alpha^{-1}$.
    In that case, we write $R_\alpha[t^{\pm 1}] = \tensoralgaut{R}{M}$ and can identify
    \begin{equation*}
        R_\alpha[t^{\pm 1}]
        \simeq \colim\left(\bigoplus_{k \ge 0} R \xrightarrow{\shift} \bigoplus_{k \ge 0} R \xrightarrow{\shift} \dots \right) 
        \simeq \bigoplus_{n \in \bbZ} R.
    \end{equation*}
    Multiplication is again given on component $n \times m$ by $R \times R \xrightarrow{\id \otimes \alpha^n} R \times R \xrightarrow{\otimes} R$.
    If $R$ is discrete, this recovers the classical ring of twisted Laurent polynomials.
\end{example}
\cref{thm:twisted_BHS} immediately specialises to the following result.

\begin{cor}
    Let $R \in \alg(\spectra)$ and $M$ a $(R, R)$-bimodule which is compact as a left $R$-module.
    Then there is a splitting
    \begin{equation*}
        E(\tensoralgaut{R}{M}) \simeq E(R)_{h\bbN} \oplus NE_M(R) \oplus \NKbar{E}_{M}(R).
    \end{equation*}
    If $M$ is induced by an automorphism $\alpha$ of $R$, this reduces to a splitting
    \begin{equation*}
        E(R_\alpha[t^{\pm 1}]) \simeq E(R)_{h\bbN} \oplus NE_\alpha(R) \oplus \NKbar{E}_\alpha(R).
    \end{equation*}
\end{cor}

\begin{example}[Generators in nonzero degree]
    Consider the action $\alpha = \Sigma^k \colon \category{C} \xrightarrow{\simeq} \category{C}$ through suspension on a perfect category $\category{C}$ for some $k \in \bbZ$.
    Note that $E(\Sigma^k) \simeq \id_{E(\category{C})}$ if $k$ is even and $E(\Sigma^k) \simeq -\id_{E(\category{C})}$ if $k$ is odd.
    From this we obtain
    \begin{equation*}
        E(\category{C})_{h \bbZ} \simeq 
        \begin{cases}
        E(\category{C}) \oplus \Sigma E(\category{C}), & $k$ \text{ even}; \\
        E(\category{C})/2, & $k$ \text{ odd}
        \end{cases}
    \end{equation*}
    for the middle term in the splitting in \cref{thm:twisted_BHS}.
    If $\category{C} = \module_R^\omega$ for some $R \in \alg(\spectra)$, then $(\module_R^\omega)_{h \bbN} \simeq \module_{R[x^{\pm 1}]}^\omega$, where $R[x^{\pm 1}]$ is the free associative algebra over $R$ with an invertible generator in degree $\lvert x \rvert = k$.

    If $\category{C}$ admits a bounded $t$-structure and $k \le 0$, then $\Sigma^{k}$ is left $t$-exact.
    \cref{thm:regularity_vanishing_nil} shows that in this case $\tau_{\ge 0}NK_{\Sigma^{k}}(\category{C}) \simeq 0$.
    It turns out that the other $\Nil$-term $\NKbar{K}_{\Sigma^{k}}(\category{C})$ is generally not trivial, not even rationally.
    As an example, consider $\category{C} = \module_R^\omega$ for a discrete ring $R$, which admits a bounded $t$-structure if $R$ is regular.
    Land-Tamme compute in \cite[Proposition 4.11]{LandTamme23}
    \begin{equation*}
        \NKbar{K}_{\Sigma^k}(R)_\bbQ \simeq 
        \begin{cases}
        \bigoplus_{n \ge 1} \Sigma^{\lvert k \vert n + 1} \mathrm{HH}(R \otimes \bbQ), & k \text{ even}; \\
        \bigoplus_{n \ge 1} \Sigma^{\lvert k \vert (2n-1) + 1} \mathrm{HH}(R \otimes \bbQ), & k \text{ odd}.
        \end{cases}    
    \end{equation*}
    For example, for $R = \bbZ$ one obtains $\mathrm{HH}(\bbZ \otimes \bbQ) \simeq \bbQ[0]$.
    This shows that the two $\Nil$-terms $NE_\alpha(\category{C})$ and $\NKbar{E}_\alpha(\category{C})$ are generally not isomorphic.
\end{example}

\subsection{\texorpdfstring{$K$-theory of mapping tori}{K-theory of mapping tori}}

As promised in the introduction, we can prove a splitting of Waldhausen's $A$-theory of mapping tori.
Let us first recall the definition.

\begin{recollect}[Waldhausen's $A$-theory]
    For $\category{C} \in \catperf$ and a space $X \in \spc$, denote by $\category{C}_X = \colim_X \category{C} \in \catperf$ the colimit over the constant $X$-shaped diagram with value $X$.
    There is an equivalence $\category{C}_X \simeq \func(X, \ind(\category{C}))^\omega$ as colimits in $\catperf$ are computed by taking compact objects in the colimit of the $\ind$-completed diagram in $\lpresentable$, or equivalently as limits in the right adjoint diagram in $\rpresentable$, and using the equivalence $X \simeq X\op$.
    We define the nonconnective $A$-theory of $X$ as $\bbA(X) = K(\spectra^\omega_X)$.
    If $X$ is connected, then $\spectra^\omega_X \simeq \module_{\sphere[\Omega X]}^\omega$ so that $\bbA(X) \simeq K(\sphere[\Omega X])$.
    The classical (finitely dominated) version of $A$-theory constructed in \cite{HKWWW01} can be recovered from this as the connective cover $A(X) = \tau_{\ge 0}\bbA(X)$.
    It differs from Waldhausen's original definition of $A$ theory in \cite{WaldhausenAKSpc1} only by $\pi_0$.
\end{recollect}

We obtain the following splitting result for mapping tori.

\begin{cor}\label{cor:K_theory_tensoring_with_space_endomorphism}
    Let $\category{C} \in \catperf$ and $X$ be a space together with a selfmap $\alpha \colon X \to X$.
    For any localising invariant $E$, there is an equivalence
    \begin{equation*}
        E(\category{C}_{X_{h\bbN}}) \simeq E(\category{C}_X)_{h\bbN} \oplus NE_\alpha(\category{C}_X) \oplus \NKbar{E}_\alpha(\category{C}_X).
    \end{equation*}
    In particular, there is a splitting 
    \begin{equation*}
        A(X_{h \bbN}) \simeq \tau_{\ge 0}(\bbA(X)_{h \bbN}) \oplus NA_{\alpha}(X) \oplus \NKbar{A}_{\alpha}(X),
    \end{equation*}
    with $\Nil$-terms $NA_{\alpha}(X) \coloneqq \tau_{\ge 0} N\bbA_{\alpha}(X)$. On 1-connective covers one has $\tau_{\ge 1}(\bbA(X)_{h \bbN}) \simeq \tau_{\ge 1}(A(X)_{h \bbN})$
\end{cor}
\begin{proof}
    The first statement is a consequence of \cref{thm:twisted_BHS} together with the equivalence $\category{C}_{X_{h\bbN}} \simeq (\category{C}_{X})_{h \bbN}$ coming from the usual composition formula for colimits.
    The statement about $A$-theory follows from this by taking $E = K$ and $\category{C} = \spectra^\omega$ and passing to connective covers.
    The equivalence $\tau_{\ge 1}(\bbA(X)_{h \bbN}) \simeq \tau_{\ge 1}(A(X)_{h \bbN})$ follows from the fact that for a span $T_1 \leftarrow T_0 \rightarrow T_2$ of spectra, the map $\tau_{\ge 0}T_1 \oplus_{\tau_{\ge 0}T_0} \tau_{\ge 0}T_2 \to T_1 \oplus_{T_0} T_2$ of pushouts is an equivalence after passing to $1$-connective covers.
\end{proof}

\begin{example}[HNN-extensions]\label{ex:hnn_extension}
    Let $G$ be a discrete group and $f \colon G \to G$ a homomorphism.
    One has $(BG)_{h \bbN} \simeq B(G \ast_f)$, where $G \ast_f$ is a generalised HNN extension, explicitly given by the presentation $G \ast_f = \langle G, t | g t = t f(g) \text{ for } g \in G\rangle$.
\end{example}

As a special case of \cref{cor:K_theory_tensoring_with_space_endomorphism} we obtain the following.

\begin{cor}
    Let $R \in \alg(\spectra)$ be a ring spectrum, $G$ a group and $f \colon G \to G$ a homomorphism.
    Then there is an equivalence
    \begin{equation*}
        E(R[G \ast_f]) \simeq E(R[G])_{h \bbN} \oplus NE_f(R[G]) \oplus \NKbar{E}_f(R[G]).
    \end{equation*}
\end{cor}
\begin{proof}
    This is a combination of \cref{cor:K_theory_tensoring_with_space_endomorphism} and \cref{ex:hnn_extension}.
\end{proof}

\subsection{\texorpdfstring{$A$-theoretic $\Nil$-terms}{A-theoretic Nil-terms}}\label{sec:NA_calculation}

The goal of this subsection is to provide a guide to the computation of certain $A$-theoretic twisted $\Nil$-terms, applying the work of B\"okstedt-Hsiang-Madsen \cite{Bokstedt-Hsiang-Madsen}. 
We will achieve this through a comparison with topological cyclic homology. 
Let us first recall the situation for $X = *$.
By work of Waldhausen \cite{WaldhausenAKSpc1}, the map 
\begin{equation*}
    NA(*)= \tau_{\ge 0} NK(\spectra^\omega)\rightarrow \tau_{\ge 0} NK(\module_\bbZ^\omega)
\end{equation*}
is a rational equivalence, but the target is rationally trivial as $\bbZ$ is a regular ring. 
On the other hand, lots is known about interesting torsion in the homotopy groups of $NK(\spectra^\omega)$ \cite{GKM08}.

For a space $X$ we denote
\[ \TC(X) \coloneqq \TC(\spectra^\omega_X). \]
From now on, we will assume that $X$ is connected and that $X$ comes with a selfequivalence $\alpha \colon X \xrightarrow{\simeq} X$.
Consider the following cube.
To shorten notation, we will write $NA^{+ / -}_\alpha(X)$ and $N\TC^{+ / -}_\alpha(X)$ for the cofiber of the respective horizontal assembly maps in the cube. 
\newlength{\perspective}
\setlength{\perspective}{2pt}
\[\begin{tikzcd}[row sep={40,between origins}, column sep={40,between origins}]
      &[-\perspective] A(X)_{h\bbZ} \ar{rr}\ar{dd}\ar{dl} &[\perspective] &[-\perspective] A(X_{h\bbZ}) \ar{dd}\ar{dl} \\[-\perspective]
    A(*)_{h\bbZ} \ar[crossing over]{rr} \ar{dd} & & A(S^1) \\[\perspective]
      & \TC(X)_{h\bbZ}  \ar{rr} \ar{dl} & &  \TC(X_{h\bbZ})\vphantom{\times_{S_1}} \ar{dl} \\[-\perspective]
    \TC(*)_{h\bbZ} \ar{rr} && \TC(S^1) \ar[from=uu,crossing over]
\end{tikzcd}\]
From \cref{cor:K_theory_tensoring_with_space_endomorphism} we see that $NA^{+ / -}_\alpha(X) \simeq NA_\alpha(\category{C}) \oplus \NKbar{A}_\alpha(\category{C})$ after passing to $1$-connective covers.
Notice that $X_{h\bbZ} \rightarrow *_{h\bbZ} \simeq S^1$ and $X \rightarrow *$ induce $\pi_0$-isomorphisms after passing to spherical group rings, so the celebrated Dundas-Goodwillie-McCarthy-theorem \cite[Theorem 7.2.2.1]{DGM} applies to show that the left face and the right face are cartesian. Thus, the whole cube is cartesian and the map
\[ \cofib(NA^{+/-}_\alpha(X) \rightarrow NA^{+/-}(*)) \rightarrow \cofib(N\TC^{+/-}_\alpha(X) \rightarrow N\TC^{+/-}(*)) \]
is an equivalence. 
Surpsisingly, the $N\TC$-terms are quite accessible for computation. 

For a space $Y$, denote by $LY = \map(S^1,Y)  $ its \textit{free loop space}. 
It carries a natural $S^1$-action. 
Now suppose that $Y$ comes with a map $f \colon Y \rightarrow S^1$. The typical example here will be a space of homotopy orbits with respect to a $\bbZ$-action. It induces a map $Lf \colon LY \rightarrow LS^1$. Note that the degree identifies $\pi_0 LS^1$ with $\bbZ$. Let us denote $L(k)Y = LY \times_{  \pi_0(LS^1) } \{k\}$ and $L(\neq k)Y = \coprod_{n\neq k} L(n)Y$.
The main result of this subsection is the following.

\begin{thm}
    \label{thm:NTC}
    For a connected based space $X$ together with a self equivalence $X \xrightarrow{\simeq} X$, there is a natural equivalence
    \[ N\TC^{+/-}_\alpha(X) \simeq \Sigma(\Sigma^\infty_+ L(\neq 0)(X_{h\bbZ}))_{hS^1} \]
    after $p$-completion at an arbitrary prime $p$.
\end{thm}

A useful observation is that \cref{thm:NTC} can be reformulated as a splitting as an infinite direct sum after $p$-completion
\[ N\TC^{+/-}_\alpha(X) \simeq  \bigoplus_{k\neq 0} \Sigma(\Sigma^\infty L(k)(X_{h\bbZ}))_{hS^1}. \]
The main ingredient is the following theorem of Bökstedt-Hsiang-Madsen \cite{Bokstedt-Hsiang-Madsen}, which can also be found in \cite[Thm. IV.3.6]{NikolausScholze18}.

\begin{thm}
    \label{thm:TC_of_loop_spaces}
    For a connected based space $X$, there is a natural pullback square
    \begin{center}
        \begin{tikzcd}
            \TC(\sphere[\Omega X]) \ar[r] \ar[d] & \Sigma(\Sigma^\infty_+ LX)_{hS^1} \ar[d]\\
            \Sigma^\infty_+ LX \ar[r, " 1 - \widetilde{\varphi_p}"] &\Sigma^\infty_+ LX
        \end{tikzcd}
    \end{center}
    after $p$-completion at an arbitrary prime $p$.
\end{thm}

In the above theorem, the map $\widetilde{\varphi_p}$ is induced by the map $LX \rightarrow LX$ which precomposes with the $p$-fold covering of $S^1$.
Our goal is to understand how the pullback in \cref{thm:TC_of_loop_spaces} behaves with taking $\bbZ$-orbits.
We will start with an easy observation about loop spaces.

\begin{lem}
    \label{lem:loops_assembly}
    Let $X$ be a space with $\bbZ$-action. Then the assembly map
    \[ (LX)_{h\bbZ} \rightarrow L(X_{h\bbZ}) \]
    can be identified with the inclusion of the path component
    \[ L(X_{h\bbZ})(0) \rightarrow L(X_{h\bbZ}). \]
\end{lem}

\begin{proof}
    We have the following commutative diagram of fibre sequences
    \begin{center}
\begin{tikzcd}
LX \arrow[d] \arrow[r] & (LX)_{h\mathbb{Z}} \arrow[r] \arrow[d] & S^1 \arrow[d] \\
LX \arrow[r]           & L(X_{h\mathbb{Z}}) \arrow[r]           & LS^1         
\end{tikzcd}
    \end{center}
    where the right vertical map is the corresponding assembly map for $X = \{*\}$. There is an equivalence $LS^1 \simeq S^1 \times \bbZ$, induced by the evaluation at $1 \in S^1$ and the degree. Under this identification, the right vertical map is the inclusion of $S^1 \times \{0\}$.
\end{proof}

We need one further observation about loop spaces for the proof of \cref{thm:NTC}. Namely, we show that the square in \cref{thm:TC_of_loop_spaces} can be simplified dramatically for spaces of the form $X_{h\bbZ}$.

\begin{lem}
    \label{lem:summand_loop_space}
    The projection to the summand $\Sigma^\infty_+LX_{h\bbZ} \rightarrow \Sigma^\infty_+LX_{h\bbZ}(0)$ induces the following pullback square.
    \begin{center}
        \begin{tikzcd}
            \Sigma^\infty_+ LX_{h\bbZ} \ar[r, " 1 - \widetilde{\varphi_p}"] \ar[d] &\Sigma^\infty_+ LX_{h\bbZ} \ar[d]\\
            \Sigma^\infty_+ LX_{h\bbZ}(0) \ar[r, " 1 - \widetilde{\varphi_p}"] &\Sigma^\infty_+ LX_{h\bbZ}(0)
        \end{tikzcd}
    \end{center}
\end{lem}

\begin{proof}
    Using the decomposition $\Sigma^\infty_+ LX_{h\bbZ} = \bigoplus_{n \in \bbZ} \Sigma^\infty_+ LX_{h\bbZ}(n)$ and the order $0 < 1 < -1 < 2 < \dots$ on $\bbZ$, the map $1-\widetilde{\varphi_p}$ is a triangular matrix whose diagonal entries are $1 - \widetilde{\varphi_p}$ at $(0,0)$ and $1$ everywhere else. Thus, $1-\widetilde{\varphi_p}$ is an equivalence when restricted to the fiber of the projection $\Sigma^\infty_+LX_{h\bbZ} \rightarrow \Sigma^\infty_+LX_{h\bbZ}(0)$.
\end{proof}

\begin{rmk}
    Note that the proof of \cref{lem:summand_loop_space} goes through even if $X_{h\bbZ}$ is replaced by some space $Y$ which merely possesses a map to $S^1$. Some generalizations are possible. For example, let $G$ be a group, whose conjugacy classes of elements admit a total order $\leq$ for which $[g] \leq [h]$ implies $[g^p] \leq [h^p]$. Suppose, $X$ comes with a map to $BG$. Consider the collection $LX(1_G)$ of components of $LX$ of loops mapping to the nullhomotopic loops in $BG$. Then we have a pullback square of the following form.
    \begin{center}
        \begin{tikzcd}
            \Sigma^\infty_+ LX \ar[r, " 1 - \widetilde{\varphi_p}"] \ar[d] &\Sigma^\infty_+ LX \ar[d]\\
            \Sigma^\infty_+ LX(1_G) \ar[r, " 1 - \widetilde{\varphi_p}"] &\Sigma^\infty_+ LX(1_G)
        \end{tikzcd}
    \end{center}
    Note that groups with a total order as required have to be $p$-torsion free. An obvious candidate for a map as above is the map $Y \rightarrow B\pi_1(Y,y)$. 
\end{rmk}

\begin{proof}[Proof of \cref{thm:NTC}]
    We start by considering the following diagram
    \setlength{\perspective}{2pt}
\begin{equation}\label{diag:na_cube}
\begin{tikzcd}[row sep={40,between origins}, column sep={40,between origins}]
      &[-\perspective] \TC(X)_{h\bbZ} \ar{rr}\ar{dd}\ar{dl} &[\perspective] &[-\perspective] (\Sigma(\Sigma_+^\infty LX)_{hS^1})_{h\bbZ} \ar{dd}\ar{dl} \\[-\perspective]
    \TC(X_{h\bbZ}) \ar[crossing over]{rr} \ar{dd} & & \Sigma(\Sigma_+^\infty LX_{h\bbZ})_{hS^1} \\[\perspective]
      & \Sigma^\infty_+ (LX)_{h\bbZ} \ar{dd} \ar{rr} \ar{dl} & &  \Sigma^\infty_+ (LX)_{h\bbZ}\vphantom{\times_{S_1}}  \ar{dd} \ar{dl} 
      \\[-\perspective]
    \Sigma^\infty_+ L(X_{h\bbZ}) \ar{dd} \ar[rr, crossing over] & & \Sigma^\infty_+ L(X_{h\bbZ}) \ar[from=uu,crossing over]
    \\[\perspective]
      & \Sigma^\infty_+ (LX)_{h\bbZ}  \ar{rr} \ar{dl} & &  \Sigma^\infty_+ (LX)_{h\bbZ}\vphantom{\times_{S_1}} \ar{dl} 
      \\[-\perspective]
    \Sigma^\infty_+ L(0)(X_{h\bbZ}) \ar{rr} & & \Sigma^\infty_+ L(0)(X_{h\bbZ}) \ar[from=uu,crossing over]
\end{tikzcd}
\end{equation}
and note that the bottom cube is cartesian, as its front and back faces are pullbacks as the vertical maps of the back face are identities and the front face is the square from \cref{lem:summand_loop_space}. The top cube is cartesian after $p$-completion, as its front and back faces are cartesian by \cref{thm:TC_of_loop_spaces}. The lowermost face of the diagram is cartesian using \cref{lem:loops_assembly}. Thus, as the outermost cube is cartesian, the top face is cartesian as well, which identifies the cofiber of
\[ \TC(X)_{h\bbZ} \rightarrow \TC(X_{h\bbZ}) \]
with the cofiber of
\[ (\Sigma(\Sigma_+^\infty LX)_{hS^1})_{h\bbZ} \rightarrow \Sigma(\Sigma_+^\infty LX_{h\bbZ})_{hS^1} . \]
Using that colimits commute together with \cref{lem:loops_assembly}, this cofiber is equivalent to that of
\[ (\Sigma(\Sigma_+^\infty (LX)_{h\bbZ})_{hS^1}) \simeq (\Sigma(\Sigma_+^\infty L(0)X)_{hS^1}) \rightarrow \Sigma(\Sigma_+^\infty LX_{h\bbZ})_{hS^1}. \]
This completes the proof.
\end{proof}

\printbibliography

\end{document}